\documentclass{amsart}
\usepackage{amssymb}
\usepackage[mathcal]{eucal}
\usepackage{upref}

\usepackage{hyperref}

\usepackage[lite,shortalphabetic]{amsrefs}

\usepackage[all,cmtip]{xy}

\DeclareMathAlphabet{\EuRm}{U}{eur}{m}{n}
\SetMathAlphabet{\EuRm}{bold}{U}{eur}{b}{n}

\makeatletter
\renewcommand\l@subsection{\@tocline{2}{0pt}{2pc}{5pc}{}}
\makeatother

\newtheoremstyle{slplain}
 {.5\baselineskip\@plus.2\baselineskip\@minus.2\baselineskip}
 {.5\baselineskip\@plus.2\baselineskip\@minus.2\baselineskip}
 {\slshape}
 {}
 {\bfseries}
 {.}
 { }
 {}

\numberwithin{equation}{section}
\theoremstyle{slplain}
\newtheorem{thm}[equation]{Theorem}  
\newtheorem{cor}[equation]{Corollary}     
\newtheorem{lem}[equation]{Lemma}         
\newtheorem{prop}[equation]{Proposition}  

\theoremstyle{definition}
\newtheorem{defn}[equation]{Definition} 

\theoremstyle{remark}
\newtheorem{rem}[equation]{Remark}      
\newtheorem{ex}[equation]{Example}      

\newcommand{\thmref}{Theorem~\ref}
\newcommand{\propref}{Proposition~\ref}

\newcommand{\lemref}{Lemma~\ref}
\newcommand{\defref}{Definition~\ref}
\newcommand{\exref}{Example~\ref}

\newcommand{\diagref}{Diagram~\ref}

\newcommand{\secref}{Section~\ref}

\newcommand*{\bs}{\boldsymbol}
\newcommand*{\bD}{\boldsymbol\Delta}
\newcommand*{\bDrest}{\boldsymbol\Delta_{\mathrm{rest}}}

\newcommand{\cat}[1]{\mathcal{#1}}

\newcommand{\op}{^{\mathrm{op}}}

\newcommand{\LKan}{\mathbf{L}}
\newcommand{\RKan}{\mathbf{R}}

\newcommand{\F}{\mathrm{F}}

\newcommand{\diag}[1]{\bs{#1}}

\DeclareMathOperator{\diagon}{diag}

\newcommand{\adj}[4]{#1\negmedspace: #2\rightleftarrows #3:\negmedspace #4}

\DeclareMathOperator{\degree}{degree}

\newcommand{\latch}{\mathrm{L}}
\newcommand{\match}{\mathrm{M}}

\newcommand{\latchcat}[2]{\mathchoice
  {\bdry\overcat{\drc{#1}}{#2}}
  {\bdry\overcat{\drc{#1}}{#2}}
  {\bdry\overcat{\drc{#1}}{#2}}
  {\bdry\overcat{\drc{#1}}{#2}}}
\newcommand{\matchcat}[2]{\mathchoice
  {\bdry\undercat{\inv{#1}}{#2}}
  {\bdry\undercat{\inv{#1}}{#2}}
  {\bdry\undercat{\inv{#1}}{#2}}
  {\bdry\undercat{\inv{#1}}{#2}}}

\newcommand{\drc}[1]{\overrightarrow{#1}}
\newcommand{\inv}[1]{\overleftarrow{#1}}

\newcommand\invfact[3]{\mathrm{Fact}_{\inv{#1}}(#2,#3)}
\newcommand\drcfact[3]{\mathrm{Fact}_{\drc{#1}}(#2,#3)}

\newcommand{\undercat}[2]{%
  {(\mathord{#2}\nonscript\,\mathord\downarrow\nonscript\,\mathord{#1})}}

\newcommand{\overcat}[2]{%
  {(\mathord{#1}\nonscript\,\mathord\downarrow\nonscript\,\mathord{#2})}}

\newcommand{\bovercat}[2]{%
  {\bigl(\mathord{#1}\nonscript\,\mathord\downarrow
             \nonscript\,\mathord{#2}\bigr)}}

\newcommand{\Top}{\mathrm{Top}}

\newcommand{\skel}[1]{\mathrm{sk}_{#1}}
\newcommand{\coskel}[1]{\mathrm{cosk}_{#1}}

\newcommand{\N}{\mathrm{N}}

\DeclareMathOperator{\Ob}{Ob}

\newcommand{\bdry}{\partial}

\DeclareMathOperator*{\colim}{colim}

\newcommand{\iso}{\approx}

\newcommand{\pushout}[3]{#1\mathbin{\mathord{\amalg}_{#2}}#3}
\newcommand{\pullback}[3]{#1\mathbin{\mathord{\times}_{\!#2}}#3}

\newcommand{\hhat}{\hat h}

\newcommand{\Comma}{\rlap{,}}
\newcommand{\Period}{\rlap{.}}
\newcommand{\Semicolon}{\rlap{;}}

\begin{document}


\title{Functors between Reedy model categories of diagrams}

\author{Philip S. Hirschhorn}

\address{Department of Mathematics\\
   Wellesley College\\
   Wellesley, Massachusetts 02481}

\email{psh@math.mit.edu}

\urladdr{http://www-math.mit.edu/\textasciitilde psh}

\author{Ismar Voli\'c}

\thanks{The second author was supported in part by the Simons
  Foundation.}

\address{Department of Mathematics\\
   Wellesley College\\
   Wellesley, Massachusetts 02481}

\email{ivolic@wellesley.edu}

\urladdr{http://ivolic.wellesley.edu}


\date{February 20, 2019}



\begin{abstract}
  If $\cat D$ is a Reedy category and $\cat M$ is a model category,
  the category $\cat M^{\cat D}$ of $\cat D$-diagrams in $\cat M$ is a
  model category under the Reedy model category structure.  If $\cat C
  \to \cat D$ is a Reedy functor between Reedy categories, then there
  is an induced functor of diagram categories $\cat M^{\cat D} \to
  \cat M^{\cat C}$.  Our main result is a characterization of the
  Reedy functors $\cat C \to \cat D$ that induce right or left Quillen
  functors $\cat M^{\cat D} \to \cat M^{\cat C}$ for every model
  category $\cat M$.  We apply these results to various situations,
  and in particular show that certain important subdiagrams of a
  fibrant multicosimplicial object are fibrant.
\end{abstract}


\maketitle

\tableofcontents

\section{Introduction}
\label{sec:intro}

The interesting functors between model categories are the \emph{left
  Quillen functors} and \emph{right Quillen functors} (see
\cite{MCATL}*{Def.~8.5.2}).
In this paper, we study Quillen functors between diagram categories
with the Reedy model category structure (see \thmref{thm:RFib}).

In more detail, if $\cat C$ is a Reedy category (see
\defref{def:ReedyCat}) and $\cat M$ is a model category, then there is
a \emph{Reedy model category structure} on the category
$\cat M^{\cat C}$ of $\cat C$-diagrams in $\cat M$ (see
\defref{def:DiagCat} and \thmref{thm:RFib}).  The original (and most
well known) examples of Reedy model category structures are the model
categories of \emph{cosimplicial objects in a model category} and of
\emph{simplicial objects in a model category} (see
\secref{sec:(multi)(co)simplicial}).

Any functor $G\colon \cat C \to \cat D$ between Reedy categories
induces a functor
\begin{displaymath}
  G^{*}\colon \cat M^{\cat D} \longrightarrow \cat M^{\cat C}
\end{displaymath}
of diagram categories (see \defref{def:InducedDiag}), and it is
important to know when such a functor $G^{*}$ is a left or a right
Quillen functor, since, for example, a right Quillen functor takes
fibrant objects to fibrant objects, and takes weak equivalences
between fibrant objects to weak equivalences (see
\propref{prop:QuillenNice}).  The results in this paper provide a
complete characterization of the Reedy functors (functors between
Reedy categories that preserve the structure; see
\defref{def:subreedy}) between diagram categories for which this is
the case for all model categories $\cat M$.

To be clear, we point out that for any Reedy functor $G\colon \cat C
\to \cat D$ there exist model categories $\cat M$ such that the
induced functor $G^{*}\colon \cat M^{\cat D} \to \cat M^{\cat C}$ is a
(right or left) Quillen functor.  For example, if $\cat M$ is a model
category in which the weak equivalences are the isomorphisms of $\cat
M$ and all maps of $\cat M$ are both cofibrations and fibrations, then
\emph{every} Reedy functor $G\colon \cat C \to \cat D$ induces a right
Quillen functor $G^{*}\colon \cat M^{\cat D} \to \cat M^{\cat C}$
(which is also a left Quillen functor).  In this paper, we
characterize those Reedy functors that induce right Quillen functors
for \emph{all} model categories $\cat M$.  More precisely, we have:

\begin{thm}
  \label{thm:GoodisGood}
  If $G\colon \cat C \to \cat D$ is a Reedy functor (see
  \defref{def:subreedy}), then the induced functor of diagram
  categories $G^{*}\colon \cat M^{\cat D} \to \cat M^{\cat C}$ is a
  right Quillen functor for every model category $\cat M$ if and only
  if $G$ is a fibering Reedy functor (see \defref{def:goodsub}).
\end{thm}

In fact, we show that if $G\colon \cat C \to \cat D$ is a Reedy
functor that is not fibering, then the induced functor of diagram
categories $G^{*}\colon \cat M^{\cat D} \to \cat M^{\cat C}$ fails to
be a right Quillen functor when $\cat M$ is the standard model
category of topological spaces (see \thmref{thm:RtGdNec}).

We also have a dual result:
\begin{thm}
  \label{thm:MainCofibering}
  If $G\colon \cat C \to \cat D$ is a Reedy functor, then the induced
  functor of diagram categories $G^{*}\colon \cat M^{\cat D} \to \cat
  M^{\cat C}$ is a left Quillen functor for every model category $\cat
  M$ if and only if $G$ is a cofibering Reedy functor (see
  \defref{def:goodsub}).
\end{thm}

In an attempt to make these results accessible to a more general
audience, we've included a description of some background material
that is well known to the experts.  The structure of the paper is as
follows: We provide some background on Reedy categories and functors
in \secref{sec:background}, including discussions of filtrations,
opposites, Quillen functors, and cofinality. The only new content for
this part is in \secref{sec:ReedyFunc}, where we define inverse and
direct $\cat C$-factorizations and (co)fibering Reedy functors, and
prove some results about them.  We then discuss several examples and
applications of \thmref{thm:GoodisGood} and
\thmref{thm:MainCofibering} in \secref{sec:Examples}.  More precisely,
we look at the subdiagrams given by truncations, diagrams defined as
skeleta, and three kinds of subdiagrams determined by (co)simplicial
and multi(co)simplicial diagrams: restricted (co)simplicial objects,
diagonals of multi(co)simplicial objects, and slices of
multi(co)simplicial objects.  We then finally present the proofs of
\thmref{thm:GoodisGood} and \thmref{thm:MainCofibering} in
\secref{sec:Proof}.  \thmref{thm:GoodisGood} will follow immediately
from \thmref{thm:FiberingThree}, which is its slight elaboration.
\thmref{thm:MainCofibering} can be proved by dualizing the proof of
\thmref{thm:GoodisGood}, but we will instead derive it in
\secref{sec:PrfCofibering} from \thmref{thm:GoodisGood} and a careful
discussion of opposite categories.

Lastly, it should be noted that, upon completing this paper, the authors learned that Theorem \ref{thm:GoodisGood} also appears as \cite[Theorem 3.22]{barwick}.  However, the methods presented in this paper are different, and the proof that appears here is more elementary.  This paper additionally provides examples  that make the material digestible for the
reader, as well as a number of applications.  In particular, we
study how the main result applies to the various subdiagrams of multicosimplicial
objects (restricted multicosimplicial objects, diagonals of
multicosimplicial objects, and slices of multicosimplicial objects)
that figure heavily in recent applications of functor calculus to the
study of links.

\section{Reedy model category structures}
\label{sec:background}

In this section, we give the definitions and results needed for the
statements and proofs of our theorems.  We assume the reader is
familiar with the basic language of model categories.  The material
here is standard, with the exception of \secref{sec:ReedyFunc} where
the key notions for characterizing Quillen functors between Reedy
model categories are introduced (\defref{def:CFactors} and
\defref{def:goodsub}).

\subsection{Reedy categories and their diagram categories}
\label{sec:ReedyCat}

\begin{defn}
  \label{def:ReedyCat}
  A \emph{Reedy category} is a small category $\cat C$ together with
  two subcategories $\drc{\cat C}$ (the \emph{direct subcategory}) and
  $\inv{\cat C}$ (the \emph{inverse subcategory}), both of which
  contain all the objects of $\cat C$, in which every object can be
  assigned a nonnegative integer (called its \emph{degree}) such that
  \begin{enumerate}
  \item Every non-identity map of $\drc{\cat C}$ raises degree.
  \item Every non-identity map of $\inv{\cat C}$ lowers degree.
  \item Every map $g$ in $\cat C$ has a unique factorization $g = \drc
    g \inv g$ where $\drc g$ is in $\drc{\cat C}$ and $\inv g$ is in
    $\inv{\cat C}$.
  \end{enumerate}
\end{defn}

\begin{rem} The function that assigns to every object of a Reedy
  category its degree is not a part of the structure, but we will
  generally assume that such a degree function has been chosen.
\end{rem}

\begin{defn}
  \label{def:DiagCat}
  Let $\cat C$ be a Reedy category and let $\cat M$ be a model
  category.
  \begin{enumerate}
  \item A \emph{$\cat C$-diagram in $\cat M$} is a functor from $\cat
    C$ to $\cat M$.
  \item The category $\cat M^{\cat C}$ of $\cat C$-diagrams in $\cat
    M$ is the category with objects the functors from $\cat C$ to
    $\cat M$ and with morphisms the natural transformations of such
    functors.
  \end{enumerate}
\end{defn}

In order to describe the \emph{Reedy model category structure} on the
diagram category $\cat M^{\cat C}$ in \thmref{thm:RFib}, we first
define the \emph{latching maps} and
\emph{matching maps} of a $\cat C$-diagram in $\cat M$ as follows.

\begin{defn}
  \label{def:Matchobj}
  Let $\cat C$ be a Reedy category, let $\cat M$ be a model category,
  let $\diag X$ and $\diag Y$ be $\cat C$-diagrams in $\cat M$, let
  $f\colon \diag X \to \diag Y$ be a map of diagrams, and let $\alpha$
  be an object of $\cat C$.
  \begin{enumerate}
  \item The \emph{latching category} $\latchcat{\cat C}{\alpha}$ of
    $\cat C$ at $\alpha$ is the full subcategory of
    $\overcat{\drc{\cat C}}{\alpha}$ (the category of objects of
    $\drc{\cat C}$ over $\alpha$; see \cite{MCATL}*{Def.~11.8.1})
    containing all of the objects except the identity map of $\alpha$.
  \item The \emph{latching object} of $\diag X$ at $\alpha$ is
    \begin{displaymath}
      \latch_{\alpha}\diag X =
      \colim_{\latchcat{\cat C}{\alpha}} \diag X
    \end{displaymath}
    and the \emph{latching map} of $\diag X$ at $\alpha$ is the
    natural map
    \begin{displaymath}
      \latch_{\alpha}\diag X \longrightarrow
      \diag X_{\alpha} \Period
    \end{displaymath}
    We will use $\latch_{\alpha}^{\cat C}\diag X$ to denote the
    latching object if the indexing category is not obvious.
  \item The \emph{relative latching map} of $f\colon \diag X \to \diag
    Y$ at $\alpha$ is the natural map
    \begin{displaymath}
      \pushout{\diag X_{\alpha}}{\latch_{\alpha} \diag X}
      {\latch_{\alpha} \diag Y} \longrightarrow
      \diag Y_{\alpha} \Period
    \end{displaymath}
  \item The \emph{matching category} $\matchcat{\cat C}{\alpha}$ of
    $\cat C$ at $\alpha$ is the full subcategory of
    $\undercat{\inv{\cat C}}{\alpha}$ (the category of objects of
    $\inv{\cat C}$ under $\alpha$; see \cite{MCATL}*{Def.~11.8.3})
    containing all of the objects except the identity map of $\alpha$.
  \item The \emph{matching object} of $\diag X$ at $\alpha$ is
    \begin{displaymath}
      \match_{\alpha}\diag X =
      \lim_{\matchcat{\cat C}{\alpha}} \diag X
    \end{displaymath}
    and the \emph{matching map} of $\diag X$ at $\alpha$ is the
    natural map
    \begin{displaymath}
      \diag X_{\alpha} \longrightarrow \match_{\alpha}\diag X \Period
    \end{displaymath}
    We will use $\match_{\alpha}^{\cat C}\diag X$ to denote the
    matching object if the indexing category is not obvious.
  \item The \emph{relative matching map} of $f\colon \diag X \to \diag
    Y$ at $\alpha$ is the map
    \begin{displaymath}
      \diag X_{\alpha} \longrightarrow
      \pullback{\diag Y_{\alpha}}{\match_{\alpha}\diag Y}
      \match_{\alpha}{\diag X} \Period
    \end{displaymath}
  \end{enumerate}
\end{defn}

\begin{thm}[\cite{MCATL}*{Def.~15.3.3 and Thm.~15.3.4}]
  \label{thm:RFib}
  Let $\cat C$ be a Reedy category and let $\cat M$ be a model
  category.  There is a model category structure on the category $\cat
  M^{\cat C}$ of $\cat C$-diagrams in $\cat M$, called the \emph{Reedy
    model category structure}, in which a map $f\colon \diag X \to
  \diag Y$ of $\cat C$-diagrams in $\cat M$ is
  \begin{itemize}
  \item a \emph{weak equivalence} if for every object $\alpha$ of
    $\cat C$ the map $f_{\alpha}\colon \diag X_{\alpha} \to \diag
    Y_{\alpha}$ is a weak equivalence in $\cat M$,
  \item a \emph{cofibration} if for every object $\alpha$ of $\cat C$
    the relative latching map $\pushout{\diag
      X_{\alpha}}{\latch_{\alpha}\diag X}{\latch_{\alpha}\diag Y} \to
    \diag Y_{\alpha}$ (see \defref{def:Matchobj}) is a cofibration in
    $\cat M$, and
  \item a \emph{fibration} if for every object $\alpha$ of $\cat C$
    the relative matching map $\diag X_{\alpha} \to \pullback{\diag
      Y_{\alpha}}{\match_{\alpha} \diag Y}{\match_{\alpha} \diag X}$
    (see \defref{def:Matchobj}) is a fibration in $\cat M$.
  \end{itemize}
\end{thm}

We also record the following standard result, which can also be
obtained from the Yoneda lemma (see \cite{McL:categories}*{p.~61}); we
will have use for it in the proof of \propref{prop:MatchIso}.
\begin{prop}
  \label{prop:DetectIso}
  If $\cat M$ is a category and $f\colon X \to Y$ is a map in $\cat
  M$, then $f$ is an isomorphism if and only if it induces an
  isomorphism of the sets of maps $f_{*}\colon \cat M(W,X) \to \cat
  M(W,Y)$ for every object $W$ of $\cat M$.
\end{prop}

\begin{proof}
  If $g\colon Y \to X$ is an inverse for $f$, then $g_{*}\colon \cat
  M(W,Y) \to \cat M(W,X)$ is an inverse for $f_{*}$.

  Conversely, if $f_{*}\colon \cat M(W,X) \to \cat M(W,Y)$ is an
  isomorphism for every object $W$ of $\cat M$, then $f_{*}\colon \cat
  M(Y,X) \to \cat M(Y,Y)$ is an epimorphism, and so there is a map
  $g\colon Y \to X$ such that $fg = 1_{Y}$.  We then have two maps
  $gf,1_{X}\colon X \to X$, and
  \begin{displaymath}
    f_{*}(gf) = fgf = 1_{Y}f = f = f_{*}(1_{X}) \Period
  \end{displaymath}
  Since $f_{*}\colon \cat M(X,X) \to \cat M(X,Y)$ is a monomorphism,
  this implies that $gf = 1_{X}$.
\end{proof}

\subsection{Filtrations of Reedy categories}
\label{sec:ReedyFilt}

The notion of a filtration of a Reedy category will be used in the
proof of \thmref{thm:RtGdNec}.

\begin{defn}
  \label{def:filtration}
  If $\cat C$ is a Reedy category (with a chosen degree function) and
  $n$ is a nonnegative integer, the \emph{$n$'th filtration}
  $\F^{n}\cat C$ of $\cat C$ (also called the \emph{$n$'th truncation}
  $\cat C^{\le n}$ of $\cat C$) is the full subcategory of $\cat C$
  with objects the objects of $\cat C$ of degree at most $n$.
\end{defn}

The following is a direct consequence of the definitions.

\begin{prop}
  \label{prop:SeqFilt}
  If $\cat C$ is a Reedy category then each of its filtrations
  $\F^{n}\cat C$ is a Reedy category with $\drc{\F^{n}\cat C} =
  \drc{\cat C} \cap \F^{n}\cat C$ and $\inv{\F^{n}\cat C} = \inv{\cat
    C} \cap \F^{n}\cat C$, and $\cat C$ equals the union of the
  increasing sequence of subcategories $\F^{0}\cat C \subset
  \F^{1}\cat C \subset \F^{2}\cat C \subset \cdots$.
\end{prop}

The following will be used in the proof of \thmref{thm:RtGdNec} (which
is one direction of \thmref{thm:GoodisGood}).

\begin{prop}[\cite{MCATL}*{Thm.~15.2.1 and Cor.~15.2.9}]
  \label{prop:ConstructFilt}
  For $n > 0$, extending a diagram $\diag X$ on $\F^{n-1}\cat D$ to
  one on $\F^{n}\cat D$ consists of choosing, for every object
  $\gamma$ of degree $n$, an object $\diag X_{\gamma}$ and a
  factorization $\latch_{\gamma} \diag X \to \diag X_{\gamma} \to
  \match_{\gamma}\diag X$ of the natural map $\latch_{\gamma} \diag X
  \to \match_{\gamma}\diag X$ from the latching object of $\diag X$ at
  $\gamma$ to the matching object of $\diag X$ at $\gamma$.
\end{prop}

\subsection{Reedy functors}
\label{sec:ReedyFunc}

In \defref{def:subreedy} we introduce the notion of a \emph{Reedy
  functor} between Reedy categories; this is a functor that preserves
the Reedy structure.

\begin{defn}
  \label{def:subreedy}
  If $\cat C$ and $\cat D$ are Reedy categories, then a \emph{Reedy
    functor} $G\colon \cat C \to \cat D$ is a functor that takes
  $\drc{\cat C}$ into $\drc{\cat D}$ and takes $\inv{\cat C}$ into
  $\inv{\cat D}$.  If $\cat D$ is a Reedy category, then a \emph{Reedy
    subcategory} of $\cat D$ is a subcategory $\cat C$ of $\cat D$
  that is a Reedy category for which the inclusion functor $\cat C \to
  \cat D$ is a Reedy functor.
\end{defn}

Note that a Reedy functor is \emph{not} required to respect the
filtrations on the Reedy categories $\cat C$ and $\cat D$ (see
\defref{def:filtration}).  Thus, a Reedy functor might take
non-identity maps to identity maps (see, e.g.,
\propref{prop:MatchIso}).

\begin{defn}
  \label{def:InducedDiag}
  If $G\colon \cat C \to \cat D$ is a Reedy functor between Reedy
  categories and $\cat M$ is a model category, then $G$ induces a
  functor of diagram categories $G^{*}\colon \cat M^{\cat D} \to \cat
  M^{\cat C}$ under which
  \begin{itemize}
  \item a functor $\diag X\colon \cat D \to \cat M$ goes to the
    functor $G^{*}\diag X\colon \cat C \to \cat M$ that is the
    composition
    $\cat C \xrightarrow{G} \cat D \xrightarrow{\diag X} \cat M$ (so
    that for an object $\alpha$ of $\cat C$ we have
    $(G^{*}\diag X)_{\alpha} = \diag X_{G\alpha}$) and
  \item a natural transformation of $\cat D$-diagrams
    $f\colon \diag X \to \diag Y$ goes to the natural transformation
    of $\cat C$-diagrams $G^{*}f$ that on an object $\alpha$ of
    $\cat C$ is the map
    $f_{G\alpha}\colon \diag X_{G\alpha} \to \diag Y_{G\alpha}$ in
    $\cat M$.
  \end{itemize}
\end{defn}

The main results of this paper (\thmref{thm:GoodisGood} and
\thmref{thm:MainCofibering}) determine when the functor $G^{*}\colon
\cat M^{\cat D} \to \cat M^{\cat C}$ is either a left Quillen functor
or a right Quillen functor for all model categories $\cat M$.  The
characterizations will depend on the notions of the \emph{category of
  inverse $\cat C$-factorizations} of a map in $\inv{\cat D}$ and the
\emph{category of direct $\cat C$-factorizations} of a map in
$\drc{\cat D}$.

\begin{defn}
  \label{def:CFactors}
  Let $G\colon \cat C \to \cat D$ be a Reedy functor between Reedy
  categories, let $\alpha$ be an object of $\cat C$, and let $\beta$
  be an object of $\cat D$.
  \begin{enumerate}
  \item If $\sigma\colon G\alpha \to \beta$ is a map in $\inv{\cat
      D}$, then the \emph{category of inverse $\cat C$-factorizations}
    of $(\alpha, \sigma)$ is the category $\invfact{\cat
      C}{\alpha}{\sigma}$ in which
    \begin{itemize}
    \item an object is a pair
      \begin{displaymath}
        \bigl((\nu\colon \alpha \to \gamma),
        (\mu\colon G\gamma\to \beta)\bigr)
      \end{displaymath}
      consisting of a non-identity map $\nu\colon \alpha \to \gamma$
      in $\inv{\cat C}$ and a map $\mu\colon G\gamma \to \beta$ in
      $\inv{\cat D}$ such that the diagram
      \begin{displaymath}
        \xymatrix@=.6em{
          {G\alpha} \ar[rr]^{G\nu} \ar[dr]_{\sigma}
          && {G\gamma} \ar[dl]^{\mu}\\
          & {\beta}
        }
      \end{displaymath}
      commutes, and
    \item a map from $\bigl((\nu\colon \alpha \to \gamma), (\mu\colon
      G\gamma\to \beta)\bigr)$ to $\bigl((\nu'\colon \alpha \to
      \gamma'), (\mu'\colon G\gamma'\to \beta)\bigr)$ is a map
      $\tau\colon \gamma \to \gamma'$ in $\inv{\cat C}$ such that the
      triangles
        \begin{displaymath}
          \vcenter{
            \xymatrix@=.6em{
              &{\alpha} \ar[dl]_{{\nu}} \ar[dr]^{\nu'}\\
              {\gamma} \ar[rr]_{\tau}
              && {\gamma'}
            }
          }
          \qquad\text{and}\qquad
          \vcenter{
            \xymatrix@=.6em{
              {G\gamma} \ar[rr]^{G\tau} \ar[dr]_{\mu}
              && {G\gamma'} \ar[dl]^{\mu'}\\
              & {\beta}
            }
          }
        \end{displaymath}
        commute.
    \end{itemize}
    We will often refer just to the map $\sigma$ when the object
    $\alpha$ is obvious.  In particular, when $G\colon \cat C \to \cat
    D$ is the inclusion of a subcategory the object $\alpha$ is
    determined by the morphism $\sigma$, and we will often refer to
    the \emph{category of inverse $\cat C$-factorizations of $\sigma$}.
  \item If $\sigma\colon \beta \to G\alpha$ is a map in $\drc{\cat
      D}$, then the \emph{category of direct $\cat C$-factorizations}
    of $(\alpha, \sigma)$ is the category $\drcfact{\cat
      C}{\alpha}{\sigma}$ in which
    \begin{itemize}
    \item an object is a pair
      \begin{displaymath}
        \bigl((\nu\colon \gamma \to \alpha),
        (\mu\colon \beta \to G\gamma)\bigr)
      \end{displaymath}
      consisting of a non-identity map $\nu\colon \gamma \to \alpha$
      in $\drc{\cat C}$ and a map $\mu\colon \beta \to G\gamma$ in
      $\drc{\cat D}$ such that the diagram
      \begin{displaymath}
        \xymatrix@=.6em{
          {\beta} \ar[rr]^{\mu} \ar[dr]_{\sigma}
          && {G\gamma} \ar[dl]^{G\nu}\\
          & {G\alpha}
        }
      \end{displaymath}
      commutes, and
    \item a map from $\bigl((\nu\colon \gamma \to \alpha), (\mu\colon
      \beta \to G\gamma)\bigr)$ to $\bigl((\nu'\colon \gamma' \to
      \alpha), (\mu'\colon \beta \to G\gamma')\bigr)$ is a map
      $\tau\colon \gamma \to \gamma'$ in $\drc{\cat C}$ such that the
      triangles
      \begin{displaymath}
        \vcenter{
          \xymatrix@=.8em{
            & {\alpha}\\
            {\gamma} \ar[rr]_{\tau} \ar[ur]^{\nu}
            && {\gamma'} \ar[ul]_{\gamma'}
          }
        }
        \qquad\text{and}\qquad
        \vcenter{
          \xymatrix@=.8em{
            {G\gamma} \ar[rr]^{G\tau}
            && {G\gamma'}\\
            & {\beta} \ar[ul]^{\mu} \ar[ur]_{\mu'}
          }
        }
      \end{displaymath}
      commute.
    \end{itemize} 
    We will often refer just to the map $\sigma$ when the object
    $\alpha$ is obvious.  In particular, when $G\colon \cat C \to \cat
    D$ is the inclusion of a subcategory the object $\alpha$ is
    determined by the morphism $\sigma$, and we will often refer to
    the \emph{category of direct $\cat C$-factorizations of $\sigma$}.
  \end{enumerate}
\end{defn}

\begin{prop}
  \label{prop:CategoryOfFacts}
  Let $G\colon \cat C \to \cat D$ be a Reedy functor between Reedy
  categories, let $\alpha$ be an object of $\cat C$, and let $\beta$
  be an object of $\cat D$.
  \begin{enumerate}
  \item If $\sigma\colon G\alpha \to \beta$ is a map in $\inv{\cat
      D}$, then we have an induced functor
    \begin{displaymath}
      G_{*}\colon \matchcat{\cat C}{\alpha} \longrightarrow
      \undercat{\inv{\cat D}}{G\alpha}
    \end{displaymath}
    from the matching category of $\cat C$ at $\alpha$ to the category
    of objects of $\inv{\cat D}$ under $G\alpha$ that takes the object
    $\alpha \to \gamma$ of $\matchcat{\cat C}{\alpha}$ to the object
    $G\alpha \to G\gamma$ of $\undercat{\inv{\cat D}}{G\alpha}$, and
    the category $\invfact{\cat C}{\alpha}{\sigma}$ of inverse $\cat
    C$-factorizations of $(\alpha, \sigma)$ (see
    \defref{def:CFactors}) is the category $\overcat{G_{*}}{\sigma}$
    of objects of $\matchcat{\cat C}{\alpha}$ over $\sigma$.
  \item If $\sigma\colon \beta \to G\alpha$ is a map in $\drc{\cat
      D}$, then we have an induced functor
    \begin{displaymath}
      G_{*}\colon \latchcat{\cat C}{\alpha} \longrightarrow
      \overcat{\drc{\cat D}}{G\alpha}
    \end{displaymath}
    from the latching category of $\cat C$ at $\alpha$ to the category
    of objects of $\drc{\cat D}$ over $G\alpha$ that takes the object
    $\gamma \to \alpha$ of $\latchcat{\cat C}{\alpha}$ to the object
    $G\gamma \to G\alpha$ of $\overcat{\drc{\cat D}}{G\alpha}$, and
    the category $\drcfact{\cat C}{\alpha}{\sigma}$ of direct $\cat
    C$-factorizations of $(\alpha, \sigma)$ is the category
    $\undercat{G_{*}}{\sigma}$ of objects of $\latchcat{\cat
      C}{\alpha}$ under $\sigma$.
  \end{enumerate}
\end{prop}

\begin{proof}
  We will prove part~1; the proof of part~2 is similar.  An object of
  $\overcat{G_{*}}{\sigma}$ is a pair $\bigl((\nu\colon \alpha \to
  \gamma), (\mu\colon G\gamma \to \beta)\bigr)$ where $\nu\colon
  \alpha \to \gamma$ is an object of $\matchcat{\cat C}{\alpha}$ and
  $\mu\colon G\gamma \to \beta$ is a map in $\inv{\cat D}$ that makes
  the triangle
  \begin{displaymath}
    \xymatrix@=.6em{
      &{G\alpha} \ar[dl]_{{G\nu}} \ar[dr]^{\sigma}\\
      {G\gamma} \ar[rr]_{\mu}
      && {\beta}
    }
  \end{displaymath}
  commute.  A map from $\bigl((\nu\colon \alpha \to \gamma),
  (\mu\colon G\gamma \to \beta)\bigr)$ to $\bigl((\nu'\colon \alpha
  \to \gamma'), (\mu'\colon G\gamma' \to \beta)\bigr)$ is a map
  $\tau\colon \gamma \to \gamma'$ in $\inv{\cat C}$ that makes the
  triangles
  \begin{displaymath}
    \vcenter{
      \xymatrix@=.8em{
        &{\alpha} \ar[dl]_{\nu} \ar[dr]^{\nu'}\\
        {\gamma} \ar[rr]_{\tau}
        && {\gamma'}
      }
    }
    \qquad\text{and}\qquad
    \vcenter{
      \xymatrix@=.6em{
        {G\gamma} \ar[rr]^{G\tau} \ar[dr]_{\mu}
        && {G\gamma'} \ar[dl]^{\mu'}\\
        & {\beta}
      }
    }
  \end{displaymath}
  commute. This is exactly the definition of the category of inverse
  $\cat C$-factorizations of $(\alpha, \sigma)$.
\end{proof}

\begin{prop}
  \label{prop:CatFacts}
  Let $\cat C$ and $\cat D$ be Reedy categories, let $G\colon \cat C
  \to \cat D$ be a Reedy functor, and let $\alpha$ be an object of
  $\cat C$.
  \begin{enumerate}
  \item If $G$ takes every non-identity map $\alpha \to \gamma$ in
    $\inv{\cat C}$ to a non-identity map in $\inv{\cat D}$, then there
    is an induced functor of matching categories
    \begin{displaymath}
      G_{*}\colon \matchcat{\cat C}{\alpha}
      \to \matchcat{\cat D}{G\alpha}
    \end{displaymath}
    (see \defref{def:Matchobj}) that takes the object
    $\eta\colon \alpha \to \gamma$ of $\matchcat{\cat C}{\alpha}$ to
    the object $G\eta\colon G\alpha \to G\gamma$ of
    $\matchcat{\cat D}{G\alpha}$.  If $\beta$ is an object of $\cat D$
    and $\sigma\colon G\alpha \to \beta$ is a map in $\inv{\cat D}$,
    then the category $\invfact{\cat C}{\alpha}{\sigma}$ of inverse
    $\cat C$-factorizations of $(\alpha, \sigma)$ (see
    \defref{def:CFactors}) is the category $\overcat{G_{*}}{\sigma}$
    of objects of $\matchcat{\cat C}{\alpha}$ over $\sigma$.
  \item If $G$ takes every non-identity map $\gamma \to \alpha$ in
    $\drc{\cat C}$ to a non-identity map in $\drc{\cat D}$, then there
    is an induced functor of latching categories
    \begin{displaymath}
      G_{*}\colon \latchcat{\cat C}{\alpha}
      \to \latchcat{\cat D}{G\alpha}
    \end{displaymath}
    (see \defref{def:Matchobj}) that takes the object
    $\eta\colon \gamma \to \alpha$ of $\latchcat{\cat C}{\alpha}$ to
    the object $G\eta\colon G\gamma \to G\alpha$ of
    $\latchcat{\cat D}{G\alpha}$.  If $\beta$ is an object of $\cat D$
    and $\sigma\colon \beta \to G\alpha$ is a map in $\drc{\cat D}$,
    then the category $\drcfact{\cat C}{\alpha}{\sigma}$ of direct
    $\cat C$-factorizations of $(\alpha, \sigma)$ is the category
    $\undercat{G_{*}}{\sigma}$ of objects of
    $\latchcat{\cat C}{\alpha}$ under $\sigma$.
  \end{enumerate}
\end{prop}

\begin{proof}
  This is identical to the proof of \propref{prop:CategoryOfFacts},
  except that the requirement that certain non-identity maps go to
  non-identity maps ensures (in part~1) that the functor $G_{*}\colon
  \matchcat{\cat C}{\alpha} \to \undercat{\inv{\cat D}}{G\alpha}$
  factors through the subcategory $\matchcat{\cat D}{G\alpha}$ of
  $\undercat{\inv{\cat D}}{G\alpha}$ and (in part~2) that the functor
  $G_{*}\colon \latchcat{\cat C}{\alpha} \to \overcat{\drc{\cat
      D}}{G\alpha}$ factors through the subcategory $\latchcat{\cat
    D}{G\alpha}$ of $\overcat{\drc{\cat D}}{G\alpha}$.
\end{proof}

The following is the main definition of this section; it is used in
the statements of our main theorems (\thmref{thm:GoodisGood} and
\thmref{thm:MainCofibering}).

\begin{defn}
  \label{def:goodsub}
  Let $G\colon \cat C \to \cat D$ be a Reedy functor between Reedy
  categories.
  \begin{enumerate}
  \item The Reedy functor $G$ is a \emph{fibering Reedy functor} if
    for every object $\alpha$ in $\cat C$, every object $\beta$ in
    $\cat D$, and every map $\sigma\colon G\alpha \to \beta$ in
    $\inv{\cat D}$, the nerve of $\invfact{\cat C}{\alpha}{\sigma}$,
    the category of inverse $\cat C$-factorizations of $(\alpha,
    \sigma)$, (see \defref{def:CFactors}) is either empty or
    connected.

    If $\cat C$ is a Reedy subcategory of $\cat D$ and if the
    inclusion is a fibering Reedy functor, then we will call $\cat C$
    a \emph{fibering Reedy subcategory} of $\cat D$.
  \item The Reedy functor $G$ is a \emph{cofibering Reedy functor} if
    for every object $\alpha$ in $\cat C$, every object $\beta$ in
    $\cat D$, and every map $\sigma\colon \beta \to G\alpha$ in
    $\drc{\cat D}$, the nerve of $\drcfact{\cat C}{\alpha}{\sigma}$,
    the category of direct $\cat C$-factorizations of $(\alpha,
    \sigma)$, (see \defref{def:CFactors}) is either empty or
    connected.

    If $\cat C$ is a Reedy subcategory of $\cat D$ and if the
    inclusion is a cofibering Reedy functor, then we will call $\cat
    C$ a \emph{cofibering Reedy subcategory} of $\cat D$.
  \end{enumerate}
\end{defn}

Examples of fibering Reedy functors and of cofibering Reedy functors
(and of Reedy functors that are not fibering and Reedy functors that
are not cofibering) are given in \secref{sec:Examples}.

\subsection{Opposites}
\label{sec:Opposites}

The results in this section will be used in the proof of
\thmref{thm:MainCofibering}, which can be found in
\secref{sec:PrfCofibering}.

\begin{prop}
  \label{prop:OpReedy}
  If $\cat C$ is a Reedy category, then the opposite category $\cat
  C\op$ is a Reedy category in which $\drc{\cat C\op} = (\inv{\cat
    C})\op$ and $\inv{\cat C\op} = (\drc{\cat C})\op$.
\end{prop}

\begin{proof}
  A degree function for $\cat C$ will serve as a degree function for
  $\cat C\op$, and factorizations $\sigma = \tau\mu$ in $\cat C$ with
  $\mu \in \inv{\cat C}$ and $\tau \in \drc{\cat C}$ correspond to
  factorizations $\sigma\op = \mu\op \tau\op$ in $\cat C\op$ with
  $\mu\op \in (\inv{\cat C})\op = \drc{\cat C\op}$ and $\tau\op \in
  (\drc{\cat C})\op = \inv{\cat C\op}$.
\end{proof}

\begin{prop}
  \label{prop:OpReedyFunc}
  If $\cat C$ and $\cat D$ are Reedy categories, then a functor
  $G\colon \cat C \to \cat D$ is a Reedy functor if and only if its
  opposite $G\op\colon \cat C\op \to \cat D\op$ is a Reedy functor.
\end{prop}

\begin{proof}
  This follows from \propref{prop:OpReedy}.
\end{proof}

\begin{lem}
  \label{lem:FactOpFact}
  Let $G\colon \cat C \to \cat D$ be a Reedy functor between Reedy
  categories, let $\alpha$ be an object of $\cat C$, and let $\beta$
  be an object of $\cat D$.
  \begin{enumerate}
  \item If $\sigma\colon G\alpha \to \beta$ is a map in $\inv{\cat
      D}$, then the opposite of the category of inverse $\cat
    C$-factorizations of $(\alpha, \sigma)$ is the category of direct
    $\cat C\op$-factorizations of $(\alpha, \sigma\op\colon \beta \to
    G\alpha)$ in $\drc{\cat D\op}$.
  \item If $\sigma\colon \beta \to G\alpha$ is a map in $\drc{\cat
      D}$, then the opposite of the category of direct $\cat
    C$-factorizations of $(\alpha, \sigma)$ is the category of inverse
    $\cat C\op$-factorizations of $(\alpha, \sigma\op\colon G\alpha
    \to \beta)$ in $\inv{\cat D\op}$.
  \end{enumerate}
\end{lem}

\begin{proof}
  We will prove part~(1); part~(2) will then follow from applying
  part~(1) to $\sigma\op\colon G\alpha \to \beta$ in $\cat C\op$ and
  remembering that $(\cat C\op)\op = \cat C$ and $(\cat D\op)\op =
  \cat D$.

  Let $\sigma\colon G\alpha \to \beta$ be a map in $\inv{\cat D}$.
  Recall from \defref{def:CFactors} that
  \begin{itemize}
  \item an object of the category of inverse $\cat C$-factorizations
    of $(\alpha, \sigma\colon G\alpha \to \beta)$ is a pair
    \begin{displaymath}
      \bigl((\nu\colon \alpha \to \gamma),
      (\mu\colon G\gamma\to \beta)\bigr)
    \end{displaymath}
    consisting of a non-identity map $\nu\colon \alpha \to \gamma$
    in $\inv{\cat C}$ and a map $\mu\colon G\gamma \to \beta$ in
    $\inv{\cat D}$ such that the composition $G\alpha
    \xrightarrow{G\nu} G\gamma \xrightarrow{\mu} \beta$ equals
    $\sigma$, and
  \item a map from $\bigl((\nu\colon \alpha \to \gamma), (\mu\colon
    G\gamma\to \beta)\bigr)$ to $\bigl((\nu'\colon \alpha \to
    \gamma'), (\mu'\colon G\gamma'\to \beta)\bigr)$ is a map
    $\tau\colon \gamma \to \gamma'$ in $\inv{\cat C}$ such that the
    triangles
      \begin{displaymath}
        \vcenter{
          \xymatrix@=.6em{
            &{\alpha} \ar[dl]_{{\nu}} \ar[dr]^{\nu'}\\
            {\gamma} \ar[rr]_{\tau}
            && {\gamma'}
          }
        }
        \qquad\text{and}\qquad
        \vcenter{
          \xymatrix@=.6em{
            {G\gamma} \ar[rr]^{G\tau} \ar[dr]_{\mu}
            && {G\gamma'} \ar[dl]^{\mu'}\\
            & {\beta}
          }
        }
      \end{displaymath}
      commute.
  \end{itemize}
  The opposite of this category has the same objects, but
  \begin{itemize}
  \item a non-identity map $\nu\colon \alpha \to \gamma$ in $\inv{\cat
      C}$ is equivalently a non-identity map $\nu\op\colon \gamma \to
    \alpha$ in $(\inv{\cat C})\op = \drc{\cat C\op}$, and
  \item a factorization $G\alpha \xrightarrow{G\nu} G\gamma
    \xrightarrow{\mu} \beta$ of $\sigma$ such that $\mu \in \inv{\cat
      D}$ is equivalently a factorization $\beta \xrightarrow{\mu\op}
    G\gamma \xrightarrow{G\nu\op} G\alpha$ of $\sigma\op\colon \beta
    \to G\alpha$ in $(\inv{\cat D})\op = \drc{\cat D\op}$
  \end{itemize}
  Thus, the opposite category can be described as the category in
  which
  \begin{itemize}
  \item An object is a pair
    \begin{displaymath}
      \bigl((\nu\op\colon \gamma \to \alpha),
      (\mu\op\colon \beta \to G\gamma)\bigr)
    \end{displaymath}
    consisting of a non-identity map $\nu\op\colon \gamma \to \alpha$
    in $(\inv{\cat C})\op = \drc{\cat C\op}$ and a map $\mu\op\colon
    \beta \to G\gamma$ in $(\inv{\cat D})\op = \drc{\cat D\op}$ such
    that the composition $\beta \xrightarrow{\mu\op} G\gamma
    \xrightarrow{G\nu\op} G\alpha$ equals $\sigma\op$, and
  \item a map from $\bigl((\nu\op\colon \gamma \to \alpha),
    (\mu\op\colon \beta \to G\gamma)\bigr)$ to $\bigl(((\nu')\op\colon
    \gamma' \to \alpha), ((\mu')\op\colon \beta \to G\gamma')\bigr)$
    is a map $\tau\op\colon \gamma' \to \gamma$ in $(\inv{\cat C})\op
    = \drc{\cat C\op}$ such that the triangles
    \begin{displaymath}
      \vcenter{
        \xymatrix@=.6em{
          & {\alpha}\\
          {\gamma} \ar[ur]^{\nu\op}
          && {\gamma'} \ar[ll]^{\tau\op} \ar[ul]_{(\nu')\op}
        }
      }
      \qquad\text{and}\qquad
      \vcenter{
        \xymatrix@=.6em{
          {G\gamma}
          && {G\gamma'} \ar[ll]_{G\tau\op}\\
          & {\beta} \ar[ul]^{\mu\op} \ar[ur]_{(\mu')\op}
        }
      }
    \end{displaymath}
    commute.
  \end{itemize}
  This is exactly the category of direct $\cat C\op$-factorizations of
  $(\alpha, \sigma\op\colon \beta \to G\alpha)$ in $\drc{\cat D\op}$.
\end{proof}

\begin{prop}
  \label{prop:OpFiberingCofibering}
  If $G\colon \cat C \to \cat D$ is a Reedy functor between Reedy
  categories, then $G$ is a fibering Reedy functor if and only if
  $G\op\colon \cat C\op \to \cat D\op$ is a cofibering Reedy functor.
\end{prop}

\begin{proof}
  Since the nerve of a category is empty or connected if and only if
  the nerve of the opposite category is, respectively, empty or
  connected, this follows from \lemref{lem:FactOpFact}.
\end{proof}

\begin{lem}
  \label{lem:OpMatch}
  Let $\diag X$ be a $\cat C$-diagram in $\cat M$ (which can also be
  viewed as a $\cat C\op$-diagram in $\cat M\op$), and let $\alpha$ be
  an object of $\cat C$.
  \begin{enumerate}
  \item The latching object $\latch^{\cat C}_{\alpha}$ of $\diag X$ as
    a $\cat C$-diagram in $\cat M$ at $\alpha$ is the matching object
    $\match^{\cat C\op}_{\alpha}$ of $\diag X$ as a $\cat
    C\op$-diagram in $\cat M\op$ at $\alpha$, and the opposite of the
    latching map $\latch^{\cat C}_{\alpha}\diag X \to \diag X$ of
    $\diag X$ as a $\cat C$-diagram in $\cat M$ at $\alpha$ is the
    matching map $\diag X \to \latch^{\cat C}_{\alpha}\diag X =
    \match^{\cat C\op}_{\alpha}\diag X$ of $\diag X$ as a $\cat
    C\op$-diagram in $\cat M\op$ at $\alpha$.
  \item The matching object $\match^{\cat C}_{\alpha}$ of $\diag X$ as
    a $\cat C$-diagram in $\cat M$ at $\alpha$ is the latching object
    $\latch^{\cat C\op}_{\alpha}$ of $\diag X$ as a $\cat
    C\op$-diagram in $\cat M\op$ at $\alpha$, and the opposite of the
    matching map $\diag X \to \match^{\cat C}_{\alpha}\diag X$ of
    $\diag X$ as a $\cat C$-diagram in $\cat M$ at $\alpha$ is the
    latching map $ \latch^{\cat C\op}_{\alpha}\diag X = \match^{\cat
      C}_{\alpha}\diag X \to \diag X$ of $\diag X$ as a $\cat
    C\op$-diagram in $\cat M\op$ at $\alpha$.
  \end{enumerate}
\end{lem}

\begin{proof}
  We will prove part~1; part~2 then follows by applying part~1 to the
  $\cat C\op$-diagram $\diag X$ in $\cat M\op$ and remembering that
  $(\cat C\op)\op = \cat C$ and $(\cat M\op)\op = \cat M$.

  The latching object $\latch^{\cat C}_{\alpha} \diag X$ of $\diag X$
  at $\alpha$ is the colimit of the diagram in $\cat M$ with an object
  $\diag X_{\beta}$ for every non-identity map $\sigma\colon \beta \to
  \alpha$ in $\drc{\cat C}$ and a map $\mu_{*}\colon \diag X_{\beta}
  \to \diag X_{\gamma}$ for every commutative triangle
  \begin{displaymath}
    \xymatrix@=.6em{
      & {\alpha}\\
      {\beta} \ar[rr]_{\mu} \ar[ur]^-{\sigma}
      && {\gamma} \ar[ul]_-{\tau}
    }
  \end{displaymath}
  in $\drc{\cat C}$ in which $\sigma$ and $\tau$ are non-identity
  maps.  Thus, $\latch^{\cat C}_{\alpha}\diag X$ can also be described
  as the limit of the diagram in $\cat M\op$ with one object $\diag
  X_{\beta}$ for every non-identity map $\sigma\op\colon \alpha \to
  \beta$ in $(\drc{\cat C})\op = \inv{\cat C\op}$ and a map
  $(\mu\op)_{*}\colon \diag X_{\gamma} \to \diag X_{\beta}$ for every
  commutative triangle
  \begin{displaymath}
    \xymatrix@=.6em{
      & {\alpha} \ar[dl]_{\sigma\op} \ar[dr]^{\tau\op}\\
      {\beta}
      && {\gamma} \ar[ll]^{\mu\op}
    }
  \end{displaymath}
  in $(\drc{\cat C})\op = \inv{\cat C\op}$ in which $\sigma\op$ and
  $\tau\op$ are non-identity maps.  Thus, $\latch^{\cat C}_{\alpha}
  \diag X = \match^{\cat C\op}_{\alpha} \diag X$.

  The latching map $\latch^{\cat C}_{\alpha}\diag X \to \diag
  X_{\alpha}$ is the unique map in $\cat M$ such that for every
  non-identity map $\sigma\colon \beta \to \alpha$ in $\drc{\cat C}$
  the triangle
  \begin{displaymath}
    \xymatrix@=.8em{
      & {\diag X_{\alpha}}\\
      {\diag X_{\beta}} \ar[r] \ar[ur]^{\sigma_{*}}
      & {\latch^{\cat C}_{\alpha}\diag X} \ar[u]
    }
  \end{displaymath}
  commutes, and so the opposite of the latching map is the unique map
  $\diag X_{\alpha} \to \latch^{\cat C}_{\alpha}\diag X = \match^{\cat
    C\op}_{\alpha}\diag X$ in $\cat M\op$ such that for every
  non-identity map $\sigma\op\colon \alpha \to \beta$ in $(\drc{\cat
    C})\op = \inv{\cat C\op}$ the triangle
  \begin{displaymath}
    \xymatrix@=.8em{
      & {\diag X_{\alpha}} \ar[d] \ar[dl]_{(\sigma\op)_{*}}\\
      {\diag X_{\beta}}
      & {\match^{\cat C\op}_{\alpha}\diag X} \ar[l]
    }
  \end{displaymath}
  commutes, i.e., the opposite of the latching map of $\diag X$ at
  $\alpha$ in $\cat C$ is the matching map of $\diag X$ at $\alpha$ in
  $\cat C\op$.
\end{proof}

\begin{lem}
  \label{lem:OpRelMatch}
  Let $f\colon \diag X \to \diag Y$ be a map of $\cat C$-diagrams in
  $\cat M$ and let $\alpha$ be an object of $\cat C$.
  \begin{enumerate}
  \item The opposite of the relative latching map (see
    \defref{def:Matchobj}) of $f$ at $\alpha$ is the relative matching
    map of the map $f\op\colon \diag Y \to \diag X$ of $\cat
    C\op$-diagrams in $\cat M\op$ at $\alpha$.
  \item The opposite of the relative matching map (see
    \defref{def:Matchobj}) of $f$ at $\alpha$ is the relative latching
    map of the map $f\op\colon \diag Y \to \diag X$ of $\cat
    C\op$-diagrams in $\cat M\op$ at $\alpha$.
  \end{enumerate}
\end{lem}

\begin{proof}
  We will prove part~(1); part~(2) then follows by applying part~(1) to the
  map of $\cat C\op$-diagrams $f\op\colon \diag Y \to \diag X$ in
  $\cat M\op$ and remembering that $(\cat C\op)\op = \cat C$ and
  $(\cat M\op)\op = \cat M$.

  If $P = \pushout{\diag X_{\alpha}}{\latch^{\cat C}_{\alpha}\diag
    X}{\latch^{\cat C}_{\alpha}\diag Y}$, then the relative latching
  map is the unique map $P \to \diag Y_{\alpha}$ that makes the diagram
  \begin{displaymath}
    \xymatrix@=.8em{
      {\latch^{\cat C}_{\alpha}\diag X} \ar[rr] \ar[dd]
      && {\latch^{\cat C}_{\alpha}\diag Y} \ar[dd] \ar@{..>}[dl]\\
      & {P} \ar[dr]\\
      {\diag X_{\alpha}} \ar[rr] \ar@{..>}[ur]
      && {\diag Y_{\alpha}}
    }
  \end{displaymath}
  commute.  The opposite of that diagram is the diagram
  \begin{displaymath}
    \xymatrix@=.8em{
      {\match^{\cat C\op}_{\alpha}\diag X}
      && {\match^{\cat C\op}_{\alpha}\diag X} \ar[ll]\\
      & {P} \ar@{..>}[ur] \ar@{..>}[dl]\\
      {\diag X_{\alpha}} \ar[uu]
      && {\diag Y_{\alpha}} \ar[uu] \ar[ll] \ar[ul]
    }
  \end{displaymath}
  in $\cat M\op$ (see \lemref{lem:OpMatch}), in which $P =
  \pullback{\diag X_{\alpha}}{\match^{\cat C\op}_{\alpha}\diag
    X}{\match^{\cat C\op}_{\alpha}\diag Y}$, and the opposite of the
  relative latching map is the unique map in $\cat M\op$ that makes
  this diagram commute, i.e., it is the relative matching map.
\end{proof}

\begin{prop}
  \label{prop:OppositeReedy}
  If $\cat M$ is a model category and $\cat C$ is a Reedy category,
  then the opposite $(\cat M^{\cat C})\op$ of the Reedy model category
  $\cat M^{\cat C}$ (see \defref{def:DiagCat}) is naturally isomorphic
  as a model category to the Reedy model category $(\cat M\op)^{\cat
    C\op}$.
\end{prop}

\begin{proof}
  The opposite $(\cat M^{\cat C})\op$ of $\cat M^{\cat C}$ is a model
  category in which
  \begin{itemize}
  \item the cofibrations of $(\cat M^{\cat C})\op$ are the opposites
    of the fibrations of $\cat M^{\cat C}$,
  \item the fibrations of $(\cat M^{\cat C})\op$ are the opposites of
    the cofibrations of $\cat M^{\cat C}$, and
  \item the weak equivalences of $(\cat M^{\cat C})\op$ are the
    opposites of the weak equivalences of $\cat M^{\cat C}$.
  \end{itemize}
  \propref{prop:OpReedy} implies that we have a Reedy model category
  structure on $(\cat M\op)^{\cat C\op}$.  The objects and maps of
  $(\cat M^{\cat C})\op$ coincide with those of $(\cat M\op)^{(\cat
    C\op)}$, and so we need only show that the model category
  structures coincide.  This follows because the opposites of the
  objectwise weak equivalences of $\cat M^{\cat C}$ are the objectwise
  weak equivalences of $(\cat M\op)^{\cat C\op}$, and
  \lemref{lem:OpRelMatch} implies that the opposites of the
  cofibrations of $\cat M^{\cat C}$ are the fibrations of $(\cat
  M\op)^{\cat C\op}$ and that the opposites of the fibrations of $\cat
  M^{\cat C}$ are the cofibrations of $(\cat M\op)^{\cat C\op}$ (see
  \thmref{thm:RFib}).
\end{proof}

\subsection{Quillen functors}
\label{sec:QuillenFunc}

\begin{defn}
  \label{def:QuilFunc}
  Let $\cat M$ and $\cat N$ be model categories and let $\adj{G}{\cat
    M}{\cat N}{U}$ be a pair of adjoint functors.  The functor $G$ is
  a \emph{left Quillen functor} and the functor $U$ is a \emph{right
    Quillen functor} if
  \begin{itemize}
  \item the left adjoint $G$ preserves both cofibrations and trivial
    cofibrations, and
  \item the right adjoint $U$ preserves both fibrations and trivial
    fibrations.
  \end{itemize}
\end{defn}

\begin{prop}
  \label{prop:QuilFunc}
  If $\cat M$ and $\cat N$ are model categories and $\adj{G}{\cat
    M}{\cat N}{U}$ is a pair of adjoint functors, then the following
  are equivalent:
  \begin{enumerate}
  \item The left adjoint $G$ is a left Quillen functor and the right
    adjoint $U$ is a right Quillen functor.
  \item The left adjoint $G$ preserves both cofibrations and trivial
    cofibrations.
  \item The right adjoint $U$ preserves both fibrations and trivial
    fibrations.
  \end{enumerate}
\end{prop}

\begin{proof}
  This is \cite{MCATL}*{Prop.~8.5.3}.
\end{proof}

\begin{prop}
  \label{prop:QuillenNice}
  Let $\cat M$ and $\cat N$ be model categories and let $\adj{G}{\cat
    M}{\cat N}{U}$ be a pair of adjoint functors.
  \begin{enumerate}
  \item If $G$ is a left Quillen functor, then $G$ takes cofibrant
    objects of $\cat M$ to cofibrant objects of $\cat N$ and takes
    weak equivalences between cofibrant objects in $\cat M$ to weak
    equivalences between cofibrant objects of $\cat N$.
  \item If $U$ is a right Quillen functor, then $U$ takes fibrant
    objects of $\cat N$ to fibrant objects of $\cat M$ and takes weak
    equivalences between fibrant objects in $\cat N$ to weak
    equivalences between fibrant objects of $\cat M$.
  \end{enumerate}
\end{prop}

\begin{proof}
  Since left adjoints take initial objects to initial objects, if the
  left adjoint $G$ takes cofibrations to cofibrations then it takes
  cofibrant objects to cofibrant objects.  The statement about weak
  equivalences follows from \cite{MCATL}*{Cor.~7.7.2}.

  Dually, since right adjoints take terminal objects to terminal
  objects, if the right adjoint $U$ takes fibrations to fibrations
  then it takes fibrant objects to fibrant objects.  The statement
  about weak equivalences follows from \cite{MCATL}*{Cor.~7.7.2}.
\end{proof}

\begin{prop}
  \label{prop:OpQuillen}
  A functor between model categories $G\colon \cat M \to \cat N$ is a
  left Quillen functor if and only if its opposite $G\op\colon \cat
  M\op \to \cat N\op$ is a right Quillen functor.
\end{prop}

\begin{proof}
  This follows because the cofibrations and trivial cofibrations of
  $\cat M\op$ are the opposites of the fibrations and trivial
  fibrations, respectively, of $\cat M$ and the fibrations and trivial
  fibrations of $\cat M\op$ are the opposites of the cofibrations and
  trivial cofibrations, respectively, of $\cat M$ (with a similar
  statement for $\cat N$).
\end{proof}

\subsection{Cofinality}
\label{sec:Cofinality}

\begin{defn}
  \label{def:cofinal}
  Let $\cat A$ and $\cat B$ be small categories and let $G\colon \cat
  A \to \cat B$ be a functor.
  \begin{itemize}
  \item The functor $G$ is \emph{left cofinal} (or \emph{initial}) if
    for every object $\alpha$ of $\cat B$ the nerve
    $\N\overcat{G}{\alpha}$ of the overcategory $\overcat{G}{\alpha}$
    is non-empty and connected.  If in addition $G$ is the inclusion
    of a subcategory, then we will say that $\cat A$ is a \emph{left
      cofinal subcategory} (or \emph{initial subcategory}) of $\cat B$.
  \item The functor $G$ is \emph{right cofinal} (or \emph{terminal})
    if for every object $\alpha$ of $\cat B$ the nerve
    $\N\undercat{G}{\alpha}$ of the undercategory
    $\undercat{G}{\alpha}$ is non-empty and connected.  If in addition
    $G$ is the inclusion of a subcategory, then we will say that $\cat
    A$ is a \emph{right cofinal subcategory} (or \emph{terminal
      subcategory}) of $\cat B$.
  \end{itemize}
\end{defn}

For the proof of the following, see \cite{McL:categories}*{IX.3} or
\cite{MCATL}*{Thm.~14.2.5}.

\begin{thm}
  \label{thm:CofinalIso}
  Let $\cat A$ and $\cat B$ be small categories and let $G\colon \cat
  A \to \cat B$ be a functor.
  \begin{enumerate}
  \item The functor $G$ is left cofinal if and only if for every
    complete category $\cat M$ (i.e., every category in which all small
    limits exist) and every diagram $\diag X\colon \cat B \to \cat M$
    the natural map $\lim_{\cat B}\diag X \to \lim_{\cat A} G^{*}\diag
    X$ is an isomorphism.
  \item The functor $G$ is right cofinal if and only if for every
    cocomplete category $\cat M$ (i.e., every category in which all
    small colimits exist) and every diagram $\diag X\colon \cat B \to
    \cat M$ the natural map $\colim_{\cat A} G^{*}\diag X \to
    \colim_{\cat B}\diag X$ is an isomorphism.
  \end{enumerate}
\end{thm}

\section{Examples}
\label{sec:Examples}

In this section, we present various examples to illustrate
\thmref{thm:GoodisGood} and \thmref{thm:MainCofibering}.

\subsection{A Reedy functor that is not fibering}
\label{sec:notfibering}

The following is an example of a Reedy subcategory that is not
fibering.

\begin{ex}
  \label{ex:NotGood}
  Let $\cat D$ be the category
  \begin{displaymath}
    \vcenter{
      \xymatrix@=1ex{
        & {\alpha} \ar[dl]_{p} \ar[dr]^{r}\\
        {\gamma} \ar[dr]_{q}
        && {\delta} \ar[dl]^{s}\\
        & {\beta}
      }
    }
    \qquad
    \text{in which $qp = sr$.}
  \end{displaymath}
  \begin{itemize}
  \item Let $\alpha$ be of degree $2$,
  \item let $\gamma$ and $\delta$ be of degree $1$, and
  \item let $\beta$ be of degree 0.
  \end{itemize}
  $\cat D$ is then a Reedy category in which $\inv{\cat D} = \cat D$
  and $\drc{\cat D}$ has only identity maps.

  Let $\cat C$ be the full subcategory of $\cat D$ on the objects
  $\{\alpha, \gamma, \delta\}$, and let $\cat C$ have the structure of
  a Reedy category that makes it a Reedy subcategory of $\cat D$.
  Although $\cat C$ is a Reedy subcategory of $\cat D$, it is not a
  fibering Reedy subcategory because the map
  $qp\colon \alpha \to \beta$ in $\inv{\cat D}$ has only two
  factorizations in which the first map is in $\inv{\cat C}$ and is
  not an identity map and the second is in $\inv{\cat D}$, $q\circ p$
  and $s\circ r$, and neither of those factorizations maps to the
  other; thus the nerve of the category of such factorizations is
  nonempty and not connected.  \thmref{thm:GoodisGood} thus implies
  that there is a model category $\cat M$ such that the restriction
  functor $\cat M^{\cat D} \to \cat M^{\cat C}$ is not a right Quillen
  functor.  This is actually proved in \thmref{thm:RtGdNec}, which
  constructs a fibrant $\cat D$-diagram in the standard model category
  of topological spaces for which the induced $\cat C$-diagram is not
  fibrant.  For the categories $\cat C$ and $\cat D$ of
  \exref{ex:NotGood}, that $\cat D$-diagram is the functor that takes
  every object of $\cat D$ to $I$, the unit interval, and takes every
  morphism of $\cat D$ to the identity map.  This is a fibrant diagram
  because every matching map is a homeomorphism, and is thus a
  fibration.  The induced $\cat C$-diagram is not fibrant, though,
  because the matching map at $\alpha$ is the diagonal map
  $I \to I\times I$, which is not a fibration.
\end{ex}

\subsection{A Reedy functor that is not cofibering}
\label{sec:notcofibering}

\propref{prop:OpFiberingCofibering} implies that the opposite of
\exref{ex:NotGood} is a Reedy subcategory that is not cofibering.

\subsection{Truncations}

\begin{prop}
  \label{prop:TruncFib}
  If $\cat C$ is a Reedy category and $n \ge 0$, then the inclusion
  functor $G\colon \cat C^{\le n} \to \cat C$ (see
  \defref{def:filtration}) is both a fibering Reedy functor and a
  cofibering Reedy functor.
\end{prop}

\begin{proof}
  We will prove that the inclusion is a fibering Reedy functor; the
  proof that it is a cofibering Reedy functor is similar.

  If $\degree(\alpha) \le n$, then the inclusion functor $G\colon \cat
  C^{\le n} \to \cat C$ induces an isomorphism of undercategories
  $G_{*}\colon \undercat{\inv{\cat C^{\le n}}}{\alpha} \to
  \undercat{\inv{\cat C}}{\alpha}$.  Let $\sigma\colon \alpha \to
  \beta$ be a map in $\inv{\cat C}$.  If $\sigma$ is the identity map,
  then the category of inverse $\cat C$-factorizations of $\sigma$ is
  empty; if $\sigma$ is not an identity map, then the object
  $\bigl((\sigma\colon \alpha \to \beta), 1_{\beta}\bigr)$ is a
  terminal object of the category of inverse $\cat C$-factorizations
  of $\sigma$, and so the nerve of the category of inverse $\cat
  C$-factorizations of $\sigma$ is connected.  Thus, $G$ is fibering.
\end{proof}

\begin{prop}
  \label{prop:LRQuil}
  If $\cat M$ is a model category, $\cat C$ is a Reedy category, and
  $n \ge 0$, then the restriction functor $\cat M^{\cat C} \to \cat
  M^{\cat C^{\le n}}$ (see \defref{def:filtration}) is both a left
  Quillen functor and a right Quillen functor.
\end{prop}

\begin{proof}
  This follows from \propref{prop:TruncFib}, \thmref{thm:GoodisGood},
  and \thmref{thm:MainCofibering}.
\end{proof}

\propref{prop:LRQuil} extends to products of Reedy categories as
follows.

\begin{prop}
  \label{prop:TruncOne}
  If $\cat C$ and $\cat D$ are Reedy categories, $\cat M$ is a model
  category, and $n \ge 0$, then the restriction functor $\cat M^{\cat
    C \times \cat D} \to \cat M^{(\cat C^{\le n}\times \cat D)}$ (see
  \defref{def:filtration}) is both a left Quillen functor and a right
  Quillen functor.
\end{prop}

\begin{proof}
  The category $\cat M^{\cat C\times \cat D}$ of $(\cat C\times\cat
  D)$-diagrams in $\cat M$ is isomorphic as a model category to the
  category $(\cat M^{\cat D})^{\cat C}$ of $\cat C$-diagrams in $\cat
  M^{\cat D}$ (see \cite{MCATL}*{Thm.~15.5.2}), and so the result
  follows from  \propref{prop:LRQuil}.
\end{proof}

\begin{prop}
  \label{prop:TruncAll}
  If $\cat M$ is a model category, $m$ is a positive integer, and for
  $1 \le i \le m$ we have a Reedy category $\cat C_{i}$ and a
  nonnegative integer $n_{i}$, then the restriction functor
  \begin{displaymath}
    \cat M^{\cat C_{1}\times \cat C_{2}\times \cdots \cat C_{m}}
    \longrightarrow
    \cat M^{\cat C_{1}^{\le n_{1}}\times \cat C_{2}^{\le n_{2}}\times
      \cdots \cat C_{m}^{\le n_{m}}}
  \end{displaymath}
  (see \defref{def:filtration}) is both a left Quillen functor and a
  right Quillen functor.
\end{prop}

\begin{proof}
  The restriction functor is the composition of the restriction
  functors
  \begin{multline*}
    \cat M^{\cat C_{1}\times \cat C_{2}\times \cdots \cat C_{m}}
    \longrightarrow
    \cat M^{\cat C_{1}^{\le n_{1}}\times \cat C_{2}\times \cdots \cat
      C_{m}}\\
    \longrightarrow
    \cat M^{\cat C_{1}^{\le n_{1}}\times \cat C_{2}^{\le n_{2}}\times
      \cdots \cat C_{m}} \longrightarrow \cdots \longrightarrow
    \cat M^{\cat C_{1}^{\le n_{1}}\times \cat C_{2}^{\le n_{2}}\times
      \cdots \cat C_{m}^{\le n_{m}}}
  \end{multline*}
  and so the result follows from \propref{prop:TruncOne}.
\end{proof}

\subsection{Skeleta}
\label{sec:skeleta}

\begin{defn}
  \label{def:skeleton}
  Let $\cat C$ be a Reedy category, let $n \ge 0$, and let $\cat M$ be
  a model category.
  \begin{enumerate}
  \item Since $\cat M$ is cocomplete, the restriction functor $\cat
    M^{\cat C} \to \cat M^{\cat C^{\le n}}$ has a left adjoint
    $\LKan\colon \cat M^{\cat C^{\le n}} \to \cat M^{\cat C}$ (see
    \cite{borceux-I}*{Thm.~3.7.2}), and we define the
    \emph{$n$-skeleton functor} $\skel{n}\colon \cat M^{\cat C} \to
    \cat M^{\cat C}$ to be the composition
    \begin{displaymath}
      \xymatrix@C=5em{
        {\cat M^{\cat C}} \ar[r]^-{\text{restriction}}
        & {\cat M^{\cat C^{\le n}}} \ar[r]^-{\LKan}
        & {\cat M^{\cat C} \Period}
      }
    \end{displaymath}
  \item Since $\cat M$ is complete, the restriction functor $\cat
    M^{\cat C} \to \cat M^{\cat C^{\le n}}$ has a right adjoint
    $\RKan\colon \cat M^{\cat C^{\le n}} \to \cat M^{\cat C}$ (see
    \cite{borceux-I}*{Thm.~3.7.2}), and we define the
    \emph{$n$-coskeleton functor} $\coskel{n}\colon \cat M^{\cat C} \to
    \cat M^{\cat C}$ to be the composition
    \begin{displaymath}
      \xymatrix@C=5em{
        {\cat M^{\cat C}} \ar[r]^-{\text{restriction}}
        & {\cat M^{\cat C^{\le n}}} \ar[r]^-{\RKan}
        & {\cat M^{\cat C} \Period}
      }
    \end{displaymath}
  \end{enumerate}
\end{defn}

\begin{prop}
  \label{prop:SkelQuil}
  If $\cat C$ is a Reedy category, $n \ge 0$, and $\cat M$ is a model
  category, then
  \begin{enumerate}
  \item the $n$-skeleton functor $\skel{n}\colon \cat M \to \cat M$ is
    a left Quillen functor, and
  \item the $n$-coskeleton functor $\coskel{n}\colon \cat M \to \cat
    M$ is a right Quillen functor.
  \end{enumerate}
\end{prop}

\begin{proof}
  Since the restriction functor is a right Quillen functor (see
  \propref{prop:LRQuil}), its left adjoint is a left Quillen functor
  (see \propref{prop:QuilFunc}).  Since the restriction is also a left
  Quillen functor (see \propref{prop:LRQuil}), its composition with
  its left adjoint is a left Quillen functor.  Similarly, the
  composition of restriction with its right adjoint is a right Quillen
  functor.
\end{proof}

\subsection{(Multi)cosimplicial and (multi)simplicial objects}
\label{sec:(multi)(co)simplicial}

In this section we consider simplicial and cosimplicial diagrams, as
well as their multidimensional versions, $m$-cosimplicial and
$m$-simplicial diagrams (see \defref{def:Delta}).  Simplicial and
cosimplicial diagrams are standard tools in homotopy theory, while
$m$-simplicial and $m$-cosimplicial ones have seen an increase in
usage in recent years, most notably through their appearance in the
calculus of functors (see \cites{cosimplcalc, FTHoLinks}).

The important questions are whether the restrictions to various
subdiagrams of $m$-simplicial and $m$-cosimplicial diagrams are
Quillen functors (and the answer will be yes in all cases that we
consider here).  The subdiagrams we will look at are the restricted
(co)simplicial objects, diagonals of $m$-(co)simplicial objects, and
slices of $m$-(co)simplicial objects.  These are considered in
Sections \ref{sec:Restricted}, \ref{sec:diagonal}, and
\ref{sec:slice}, respectively.  In particular, the fibrancy of the
slices of a fibrant $m$-dimensional cosimplicial object is needed to
justify taking its totalization one dimension at a time, as is done in
both \cite{cosimplcalc} and \cite{FTHoLinks}.  This and some further
results about totalizations of $m$-cosimplicial objects will be
addressed in future work.

We begin by recalling the definitions:
\begin{defn}
  \label{def:Delta}
  For every nonnegative integer $n$, we let $[n]$ denote the ordered
  set $(0, 1, 2, \ldots, n)$.
  \begin{enumerate}
  \item The \emph{cosimplicial indexing category} $\bD$ is the
    category with objects the $[n]$ for $n \ge 0$ and with
    $\bD\bigl([n],[k]\bigr)$ the set of weakly monotone functions $[n]
    \to [k]$.
  \item A \emph{cosimplicial object} in a category $\cat M$ is a
    functor from $\bD$ to $\cat M$.
  \item If $m$ is a positive integer, then an \emph{$m$-cosimplicial
      object} in $\cat M$ is a functor from $\bD^{m}$ to $\cat M$.
  \item The \emph{simplicial indexing category} $\bD\op$, the opposite
    category of $\bD$.
  \item A \emph{simplicial object} in a category $\cat M$ is a functor
    from $\bD\op$ to $\cat M$.
  \item If $m$ is a positive integer, then an \emph{$m$-simplicial
      object} in $\cat M$ is a functor from $(\bD^{m})\op =
    (\bD\op)^{m}$ to $\cat M$.
  \end{enumerate}
\end{defn}

\begin{defn}
  \label{def:CosReSt}
  The \emph{standard Reedy category structure} on the cosimplicial
  indexing category $\bD$ (see \defref{def:Delta}) is the one in which
  \begin{itemize}
  \item the direct subcategory $\drc{\bD}$ consists of the injective
    functions and
  \item the inverse subcategory $\inv{\bD}$ consists of the surjective
    functions,
  \end{itemize}
  and the \emph{standard degree function} assigns the object $[n]$
  degree $n$.
\end{defn}

\subsubsection{Restricted cosimplicial objects and restricted simplicial
  objects}
\label{sec:Restricted}

For examples of fibering Reedy subcategories and cofibering Reedy
subcategories that include all of the objects, we consider the
restricted cosimplicial (or semi-cosimplicial) and restricted
simplicial (or semi-simplicial) indexing categories.

\begin{defn}
  \label{def:DeltaRest}
  For $n$ a nonnegative integer, let $[n]$ denote the ordered set
  $(0, 1, 2, \ldots, n)$.
  \begin{enumerate}
  \item The \emph{restricted cosimplicial indexing category} $\bDrest$
    is the category with objects the ordered sets $[n]$ for $n \ge 0$
    and with $\bDrest\bigl([n], [k]\bigr)$ the \emph{injective} order
    preserving maps $[n] \to [k]$.
    
    The category $\bDrest$ is thus a subcategory of $\bD$, the
    cosimplicial indexing category (see \defref{def:Delta}).  In fact,
    $\bDrest = \drc{\bD}$, the direct subcategory of $\bD$ (see
    \defref{def:CosReSt}).
  \item The \emph{restricted simplicial indexing category}
    $\bDrest\op$ is the opposite of the restricted cosimplicial
    indexing category.
  \item If $\cat M$ is a category, then a \emph{restricted
      cosimplicial object} in $\cat M$ is a functor from $\bDrest$ to
    $\cat M$.
  \item If $\cat M$ is a category, a \emph{restricted simplicial
      object} in $\cat M$ is a functor from $(\bDrest)\op$ to $\cat
    M$.
  \end{enumerate}
\end{defn}

If we let $G\colon \bDrest \to \bD$ be the inclusion, then for $\diag
X$ a cosimplicial object in $\cat M$ the induced diagram $G^{*}\diag
X$ is a restricted cosimplicial object in $\cat M$, called the
\emph{underlying restricted cosimplicial object} of $\diag X$; it is
obtained from $\diag X$ by ``forgetting the codegeneracy operators''.
Similarly, if we let $G\colon \bDrest\op \to \bD\op$ be the inclusion,
then for $\diag Y$ a simplicial object in $\cat M$ the induced diagram
$G^{*}\diag Y$ is a restricted simplicial object in $\cat M$, called
the \emph{underlying restricted simplicial object} of $\diag Y$,
obtained from $\diag Y$ by ``forgetting the degeneracy operators''.

\begin{prop}
  \label{prop:DrcFib}
  Let $\cat D$ be a Reedy category and let $\cat C = \drc{D}$, the
  direct subcategory of $\cat D$.
  \begin{enumerate}
  \item The inclusion $\cat C \to \cat D$ is both a fibering Reedy
    functor and a cofibering Reedy functor.
  \item The inclusion $\cat C\op \to \cat D\op$ is both a fibering
    Reedy functor and a cofibering Reedy functor.
  \end{enumerate}
\end{prop}

\begin{proof}
  We will prove part~1; part~2 will then follow from
  \propref{prop:OpFiberingCofibering}.

  We first prove that the inclusion $\cat C \to \cat D$ is the
  inclusion of a cofibering Reedy subcategory.  Let
  $\sigma\colon \beta \to \alpha$ be a map in $\cat D$.  If $\sigma$
  is an identity map, then the category of direct
  $\cat C$-factorizations of $\sigma$ is empty.  If $\sigma$ is not an
  identity map, then
  $\bigl((\sigma\colon \beta \to \alpha), 1_{\beta}\bigr)$ is an
  object of the category of direct $\cat C$-factorizations of $\sigma$
  that maps to every other object of that category, and so the nerve
  of that category is connected.

  We now prove that the inclusion $\cat C \to \cat D$ is the inclusion
  of a fibering Reedy subcategory.  Let
  $\sigma\colon \alpha \to \beta$ be a map in $\inv{\cat D}$.  Since
  there are no non-identity maps in $\cat C$, the category of inverse
  $\cat C$-factorizations of $\sigma$ is empty.
\end{proof}

\begin{thm}
  \label{thm:RestCosFib}
  \leavevmode
  \begin{enumerate}
  \item The inclusion $\bDrest \to \bD$ of the restricted cosimplicial
    indexing category into the cosimplicial indexing category is both
    a fibering Reedy functor and a cofibering Reedy functor.
  \item The inclusion $\bDrest\op \to \bD\op$ of the restricted
    simplicial indexing category into the simplicial indexing category
    is both a fibering Reedy functor and a cofibering Reedy functor.
  \end{enumerate}
\end{thm}

\begin{proof}
  This follows from \propref{prop:DrcFib}.
\end{proof}

\begin{thm}
  \label{thm:RestCosQuil}
  Let $\cat M$ be a model category.
  \begin{enumerate}
  \item The functor $\cat M^{\bD} \to \cat M^{\bDrest}$ that ``forgets
    the codegeneracies'' of a cosimplicial object is both a left
    Quillen functor and a right Quillen functor.
  \item The functor $\cat M^{\bD\op} \to \cat M^{\bDrest\op}$ that
    ``forgets the degeneracies'' of a simplicial object is both a left
    Quillen functor and a right Quillen functor.
  \end{enumerate}
\end{thm}

\begin{proof}
  This follows from \thmref{thm:RestCosFib}, \thmref{thm:GoodisGood},
  and \thmref{thm:MainCofibering}.
\end{proof}

\subsubsection{Diagonals of multicosimplicial and multisimplicial
  objects}
\label{sec:diagonal}

\begin{defn}
  \label{def:diag}
  Let $m$ be a positive integer.
  \begin{enumerate}
  \item The \emph{diagonal embedding} of the category $\bD$ into
    $\bD^{m}$ is the functor $D\colon \bD \to \bD^{m}$ that takes the
    object $[k]$ of $\bD$ to the object $\bigl(\underbrace{[k], [k],
      \ldots, [k]}_{\text{$m$ times}}\bigr)$ of $\bD^{m}$ and the
    morphism $\phi\colon [p] \to [q]$ of $\bD$ to the morphism
    $(\phi^{m})$ of $\bD^{m}$.
  \item If $\cat M$ is a category and $\diag X$ is an $m$-cosimplicial
    object in $\cat M$, then the \emph{diagonal} $\diagon\diag X$ of
    $\diag X$ is the cosimplicial object in $\cat M$ that is the
    composition 
    \begin{displaymath}
      \bD \xrightarrow{D} \bD^{m} \xrightarrow{\diag X} \cat M\Comma
    \end{displaymath}
    so that $(\diagon\diag X)^{k} = \diag X^{(k, k, \ldots, k)}$.
  \item If $\cat M$ is a category and $\diag X$ is an $m$-simplicial
    object in $\cat M$, then the \emph{diagonal} $\diagon\diag X$ of
    $\diag X$ is the simplicial object in $\cat M$ that is the
    composition 
    \begin{displaymath}
      \bD\op \xrightarrow{D\op} (\bD^{m})\op = (\bD\op)^{m}
      \xrightarrow{\diag X} \cat M\Comma
    \end{displaymath}
    so that $(\diagon\diag X)_{k} = \diag X_{(k, k, \ldots, k)}$.
  \end{enumerate}
\end{defn}

\begin{thm}
  \label{thm:DiagFibr}
  Let $m$ be a positive integer.
  \begin{enumerate}
  \item The diagonal embedding $D\colon \bD \to \bD^{m}$ is a fibering
    Reedy functor.
  \item The diagonal embedding $D\op\colon \bD\op \to (\bD^{m})\op =
    (\bD\op)^{m}$ is a cofibering Reedy functor.
  \end{enumerate}
\end{thm}

\begin{proof}
  We will prove part~1; part~2 will then follow from
  \propref{prop:OpFiberingCofibering}.

  We will identify $\bD$ with its image in $\bD^{m}$, so that the
  objects of $\bD$ are the $m$-tuples $\bigl([k],[k],\ldots,
  [k]\bigr)$.  If $(\alpha_{1},\alpha_{2}, \ldots, \alpha_{m})\colon
  \bigl([k],[k],\ldots, [k]\bigr) \to \bigl([p_{1}], [p_{2}], \ldots,
  [p_{m}]\bigr)$ is a map in $\inv{\bD^{m}}$, then
  \cite{diagn}*{Lem.~5.1} implies that it has a terminal factorization
  through a diagonal object of $\bD^{m}$.  If that terminal
  factorization is through the identity map of $\bigl([k],[k], \ldots,
  [k]\bigr)$, then the category of inverse $\bD$-factorizations of
  $(\alpha_{1},\alpha_{2},\ldots, \alpha_{m})$ is empty; if that
  terminal factorization is not through the identity map, then it is a
  terminal object of the category of inverse $\bD$-factorizations of
  $(\alpha_{1},\alpha_{2}, \ldots, \alpha_{m})$, and so the nerve of
  that category is connected.
\end{proof}

Part~1 of the following corollary appears in \cite{diagn}.

\begin{cor}
  \label{cor:DiagFibr}
  Let $m$ be a positive integer and let $\cat M$ be a model category.
  \begin{enumerate}
  \item The functor that takes an $m$-cosimplicial object in $\cat M$
    to its diagonal cosimplicial object is a right Quillen functor.
  \item The functor that takes an $m$-simplicial object in $\cat M$ to
    its diagonal simplicial object is a left Quillen functor.
  \end{enumerate}
\end{cor}

\begin{proof}
  This follows from \thmref{thm:DiagFibr}, \thmref{thm:GoodisGood},
  and \thmref{thm:MainCofibering}.
\end{proof}

\subsubsection{Slices of multicosimplicial and multisimplicial
  objects}
\label{sec:slice}

\begin{defn}
  \label{def:SliceCat}
  Let $n$ be a positive integer and for $1 \le i \le n$ let $\cat
  C_{i}$ be a category.  If $K$ is a subset of $\{1, 2, \ldots, n\}$,
  then a \emph{$K$-slice} of the product category $\prod_{i=1}^{n}
  \cat C_{i}$ is the category $\prod_{i \in K} \cat C_{i}$.  (If $K$
  consists of a single integer $j$, then we will use the term
  \emph{$j$-slice} to refer to the $K$-slice.)  An \emph{inclusion of
    the $K$-slice} is a functor $\prod_{i\in K} \cat C_{i} \to
  \prod_{i=1}^{n} \cat C_{i}$ defined by choosing an object
  $\alpha_{i}$ of $\cat C_{i}$ for $i \in \bigl(\{1, 2, \ldots,
  n\}-K\bigr)$ and inserting $\alpha_{i}$ into the $i$'th coordinate
  for $i \in \bigl(\{1, 2, \ldots, n\}-K\bigr)$.
\end{defn}

\begin{thm}
  \label{thm:SliceFibCofib}
  Let $n$ be a positive integer and for $1 \le i \le n$ let $\cat
  C_{i}$ be a Reedy category.  For every subset $K$ of $\{1, 2,
  \ldots, n\}$ both the product $\prod_{i=1}^{n}\cat C_{i}$ and the
  product $\prod_{i\in K}\cat C_{i}$ are Reedy categories (see
  \cite{MCATL}*{Prop.~15.1.6}), and every inclusion of a $K$-slice
  $\prod_{i\in K}\cat C_{i} \to \prod_{i=1}^{n}\cat C_{i}$ (see
  \defref{def:SliceCat}) is both a fibering Reedy functor and a
  cofibering Reedy functor.
\end{thm}

\begin{proof}
  We will show that every inclusion is a fibering Reedy functor; the
  proof that it is a cofibering Reedy functor is similar (and also
  follows from applying the fibering case to the inclusion
  $\prod_{i\in K}\cat C_{i}\op \to \prod_{i=1}^{n}\cat C_{i}\op$; see
  \propref{prop:OpFiberingCofibering}).  We will assume that $K =
  \{1,2\}$; the other cases are similar.
 
  Let $(\beta_{1}, \beta_{2}, \alpha_{3}, \alpha_{4}, \ldots,
  \alpha_{n})$ be an object of $\prod_{i\in K}\cat C_{i}$ and let
  \begin{displaymath}
    (\sigma_{1},\sigma_{2},\ldots, \sigma_{n})\colon
    (\beta_{1},\beta_{2},\alpha_{3},\alpha_{4},\ldots, \alpha_{n})
    \longrightarrow
    (\gamma_{1},\gamma_{2},\ldots,\gamma_{n})
  \end{displaymath}
  be a map in $\inv{\prod_{i=1}^{n}\cat C_{i}}$.  Since
  $\inv{\prod_{i=1}^{n}\cat C_{i}} = \prod_{i=1}^{n}\inv{\cat C_{i}}$,
  each $\sigma_{i} \in \inv{\cat C_{i}}$.  If $\sigma_{1}$ and
  $\sigma_{2}$ are both identity maps, then the category of inverse
  $\prod_{i \in K}\cat C_{i}$-factorizations of
  $(\sigma_{1},\sigma_{2},\ldots, \sigma_{n})$ is empty.  Otherwise,
  the category of inverse $\prod_{i \in K}\cat C_{i}$-factorizations
  of $(\sigma_{1},\sigma_{2},\ldots, \sigma_{n})$ contains the object
  \begin{multline*}
    (\beta_{1},\beta_{2},\alpha_{3},\alpha_{4},\ldots,\alpha_{n})
    \xrightarrow{(\sigma_{1},\sigma_{2},1_{\alpha_{3}},
      1_{\alpha_{4}}, \ldots, 1_{\alpha_{n}})}
    (\gamma_{1},\gamma_{2},\alpha_{3},\alpha_{4},\ldots,\alpha_{n})\\
    \xrightarrow{(1_{\gamma_{1}},1_{\gamma_{2}},\sigma_{3},\sigma_{4},
      \ldots, \sigma_{n})} (\gamma_{1},\gamma_{2},\ldots, \gamma_{n})
  \end{multline*}
  and every other object of the category of inverse $\prod_{i \in
    K}\cat C_{i}$-factorizations of $(\sigma_{1},\sigma_{2},\ldots,
  \sigma_{n})$ maps to this one.  Thus the nerve of the category of
  inverse $\prod_{i \in K}\cat C_{i}$-factorizations of
  $(\sigma_{1},\sigma_{2},\ldots, \sigma_{n})$ is connected.
\end{proof}

\begin{thm}
  \label{thm:GenSliceQF}
  If $\cat M$ is a model category, $n$, $\cat C_{i}$ for $1 \le i \le
  n$, and $K$ are as in \thmref{thm:SliceFibCofib}, and the functor
  $\prod_{i\in K}\cat C_{i} \to \prod_{i=1}^{n}\cat C_{i}$ is the
  inclusion of a $K$-slice, then the restriction functor
  \begin{displaymath}
    \cat M^{(\prod_{i=1}^{n}\cat C_{i})} \longrightarrow
    \cat M^{(\prod_{i \in K}\cat C_{i})}
  \end{displaymath}
  is both a left Quillen functor and a right Quillen functor.
\end{thm}

\begin{proof}
  This follows from \thmref{thm:GoodisGood},
  \thmref{thm:MainCofibering}, and \thmref{thm:SliceFibCofib}.
\end{proof}

\begin{defn}
  \label{def:slice}
  Let $\cat M$ be a model category and let $m$ be a positive integer.
  \begin{enumerate}
  \item If $\diag X$ is an $m$-cosimplicial object in $\cat M$, then a
    \emph{slice} of $\diag X$ is a cosimplicial object in $\cat M$
    defined by restricting all but one factor of $\bD^{m}$.
  \item If $\diag X$ is an $m$-simplicial object in $\cat M$, then a
    \emph{slice} of $\diag X$ is a simplicial object in $\cat M$
    defined by restricting all but one factor of $(\bD\op)^{m}$.
  \end{enumerate}
\end{defn}

\begin{thm}
  \label{thm:SliceQF}
  Let $\cat M$ be a model category and let $m$ be a positive integer.
  \begin{enumerate}
  \item The functor $\cat M^{\bD^{m}} \to \cat M^{\bD}$ that restricts
    a multicosimplicial object to a slice (see \defref{def:slice}) is
    a both a left Quillen functor and a right Quillen functor.
  \item The functor $\cat M^{(\bD\op)^{m}} \to \cat M^{\bD\op}$ that
    restricts a multisimplicial object to a slice is both a left
    Quillen functor and a right Quillen functor.
  \end{enumerate}
\end{thm}

\begin{proof}
  This follows from \thmref{thm:GenSliceQF}.
\end{proof}

\begin{cor}
  \label{cor:SliceQF}
  Let $\cat M$ be a model category and let $m$ be a positive integer.
  \begin{enumerate}
  \item If $\diag X$ is a cofibrant $m$-cosimplicial object in
    $\cat M$, then every slice of $\diag X$ is a cofibrant
    cosimplicial object.
  \item If $\diag X$ is a fibrant $m$-cosimplicial object in $\cat M$,
    then every slice of $\diag X$ is a fibrant cosimplicial object.
  \item If $\diag X$ is a cofibrant $m$-simplicial object in $\cat M$,
    then every slice of $\diag X$ is a cofibrant simplicial object.
  \item If $\diag X$ is a fibrant $m$-simplicial object in $\cat M$,
    then every slice of $\diag X$ is a fibrant simplicial object.
  \end{enumerate}
\end{cor}

\begin{proof}
  This follows from \thmref{thm:SliceQF}.  
\end{proof}

\section{Proofs of the main theorems}
\label{sec:Proof}

Our main result, \thmref{thm:GoodisGood}, will follow immediately from
\thmref{thm:FiberingThree} below (the latter is an elaboration of the
former).  The proof of its dual, \thmref{thm:MainCofibering}, will use
\thmref{thm:GoodisGood} and can be found in
\secref{sec:PrfCofibering}.

\begin{thm}
  \label{thm:FiberingThree}
  If $G\colon \cat C \to \cat D$ is a Reedy functor between Reedy
  categories, then the following are equivalent:
  \begin{enumerate}
  \item The functor $G$ is a fibering Reedy functor (see
    \defref{def:goodsub}).
  \item For every model category $\cat M$ the induced functor of
    diagram categories $G^{*}\colon \cat M^{\cat D} \to \cat M^{\cat
      C}$ is a right Quillen functor.
  \item For every model category $\cat M$ the induced functor of
    diagram categories $G^{*}\colon \cat M^{\cat D} \to \cat M^{\cat
      C}$ takes fibrant objects of $\cat M^{\cat D}$ to fibrant
    objects of $\cat M^{\cat C}$.
  \end{enumerate}
\end{thm}

\begin{proof} The proof will be completed by the proofs of
  \thmref{thm:RightQuil} and \thmref{thm:RtGdNec} below.  More
  precisely, we will have
  \begin{displaymath}
    \xymatrix@=7em{
      {(1)} \ar@{=>}[r]^{\text{\thmref{thm:RightQuil}}}
      & {(2)} \ar@{=>}[r]^{\text{\propref{prop:QuillenNice}}}
      & {(3)} \ar@{=>}[r]^{\text{\thmref{thm:RtGdNec}}}
      & {(1)}
    }\qedhere
  \end{displaymath}
\end{proof}

\begin{thm}
  \label{thm:RightQuil}
  If $G\colon \cat C \to \cat D$ is a fibering Reedy functor and $\cat
  M$ is a model category, then the induced functor of diagram
  categories $G^{*}\colon \cat M^{\cat D} \to \cat M^{\cat C}$ is a
  right Quillen functor.
\end{thm}

\begin{thm}
  \label{thm:RtGdNec}
  If $G\colon \cat C \to \cat D$ is a Reedy functor that is not a
  fibering Reedy functor, then there is a fibrant $\cat D$-diagram of
  topological spaces for which the induced $\cat C$-diagram is not
  fibrant.
\end{thm}

The proof of \thmref{thm:RightQuil} is given in
\secref{sec:ProofRightQuil}, while the proof of \thmref{thm:RtGdNec}
can be found in \secref{sec:RtGdNec}.

In summary, the proofs of our main results, \thmref{thm:GoodisGood}
and \thmref{thm:MainCofibering}, thus have the following structure:
\begin{displaymath}
  \vcenter{
    \xymatrix@C=1.5em{
      \txt{\thmref{thm:RightQuil} \\
           (\secref{sec:ProofRightQuil})}
      \ar@{=>}[dr]
      \\
      & \text{\thmref{thm:FiberingThree}} \ar@{=>}[r]
      & \text{\thmref{thm:GoodisGood}} \ar@{=>}[d]
      \\
      \txt{\thmref{thm:RtGdNec} \\
           (\secref{sec:RtGdNec})}
      \ar@{=>}[ur]
     &
     &\txt{\thmref{thm:MainCofibering} \\
           (\secref{sec:PrfCofibering})}
   }
 }
\end{displaymath}

\subsection{Proof of \thmref{thm:RightQuil}}
\label{sec:ProofRightQuil}

We work backward, first giving the proof of the main result.  The
completion of that proof will depend on two key assertions,
\propref{prop:MatchFib} and \propref{prop:MatchIso}, whose proofs are
given in Sections \ref{sec:PrfMatchFib} and \ref{sec:PrfMatchIso}.
The assumption that we have a fibering Reedy functor is used only in
the proofs of \propref{prop:MatchFib} and \propref{prop:isoprime} (the
latter is used in the proof of the former).

\begin{proof}[Proof of \thmref{thm:RightQuil}]
  Since $\cat M$ is cocomplete, the left adjoint of $G^{*}$ exists
  (see \cite{borceux-I}*{Thm.~3.7.2} or
  \cite{McL:categories}*{p.~235}).  Thus, to show that the induced
  functor $\cat M^{\cat D} \to \cat M^{\cat C}$ is a right Quillen
  functor, we need only show that it preserves fibrations and trivial
  fibrations (see \propref{prop:QuilFunc}).  Since the weak
  equivalences in $\cat M^{\cat D}$ and $\cat M^{\cat C}$ are the
  objectwise ones, any weak equivalence in $\cat M^{\cat D}$ induces a
  weak equivalence in $\cat M^{\cat C}$.  Thus, if we show that the
  induced functor preserves fibrations, then we will also know that it
  takes maps that are both fibrations and weak equivalences to maps
  that are both fibrations and weak equivalences, i.e., that it also
  preserves trivial fibrations.

  To show that the induced functor $\cat M^{\cat D} \to \cat M^{\cat
    C}$ preserves fibrations, let $\diag X \to \diag Y$ be a fibration
  of $\cat D$-diagrams in $\cat M$; we will let $G^{*}\diag X$ and
  $G^{*}\diag Y$ denote the induced diagrams on $\cat C$.  For every
  object $\alpha$ of $\cat C$, the matching objects of $\diag X$ and
  $\diag Y$ at $\alpha$ in $\cat M^{\cat C}$ are
  \begin{displaymath}
    \match_{\alpha}^{\cat C} G^{*}\diag X =
    \lim_{\matchcat{\cat C}{\alpha}} G^{*}\diag X
    \qquad\text{and}\qquad
    \match_{\alpha}^{\cat C} G^{*}\diag Y =
    \lim_{\matchcat{\cat C}{\alpha}} G^{*}\diag Y
  \end{displaymath}
  and we define $P_{\alpha}^{\cat C}$ by letting the diagram
  \begin{equation}
    \label{diag:DefPB}
    \vcenter{
      \xymatrix{
        {P_{\alpha}^{\cat C}} \ar@{..>}[r] \ar@{..>}[d]
        & {(G^{*}\diag Y)_{\alpha}} \ar[d]\\
        {\match_{\alpha}^{\cat C} G^{*}\diag X} \ar[r]
        & {\match_{\alpha}^{\cat C} G^{*}\diag Y}
      }
    }
  \end{equation}
  be a pullback; we must show that the relative matching map
  $(G^{*}\diag X)_{\alpha} \to P_{\alpha}^{\cat C}$ is a fibration
  (see \thmref{thm:RFib}), and there are two cases:
  \begin{enumerate}
  \item There is a non-identity map $\alpha \to \gamma$ in $\inv{\cat
      C}$ that $G$ takes to the identity map of $G\alpha$.
  \item $G$ takes every non-identity map $\alpha \to \gamma$ in
    $\inv{\cat C}$ to a non-identity map in $\inv{\cat D}$.
  \end{enumerate}
  In the first case, \propref{prop:MatchIso} (in
  \secref{sec:PrfMatchIso} below) implies that the pullback
  \diagref{diag:DefPB} is isomorphic to the diagram
  \begin{displaymath}
    \xymatrix{
      {P^{\cat C}_{\alpha}} \ar[r] \ar[d]
      & {(G^{*}\diag Y)_{\alpha}}
      \ar[d]^{1_{(G^{*}\diag Y)_{\alpha}}}\\
      {(G^{*}\diag X)_{\alpha}} \ar[r]
      & {(G^{*}\diag Y)_{\alpha}}
    }
  \end{displaymath}
  in which the vertical map on the left is an isomorphism $P^{\cat
    C}_{\alpha} \iso (G^{*}\diag X)_{\alpha}$.  Thus, the composition
  of the relative matching map with that isomorphism is the identity
  map of $(G^{*}\diag X)_{\alpha}$, and so the relative matching map
  is an isomorphism $(G^{*}\diag X)_{\alpha} \to P^{\cat C}_{\alpha}$,
  and is thus a fibration.

  We are left with the second case, and so we can assume that $G$
  takes every non-identity map $\alpha \to \gamma$ in $\inv{\cat C}$
  to a non-identity map in $\inv{\cat D}$.  In this case, $G$ induces
  a functor $G_{*}\colon \matchcat{\cat C}{\alpha} \to \matchcat{\cat
    D}{G\alpha}$ that takes the object $f\colon \alpha \to \gamma$ of
  $\matchcat{\cat C}{\alpha}$ to the object $Gf\colon G\alpha \to
  G\gamma$ of $\matchcat{\cat D}{G\alpha}$ (see
  \propref{prop:CatFacts}).

  The matching objects of $\diag X$ and $\diag Y$ at $G\alpha$ in
  $\cat M^{\cat D}$ are
  \begin{displaymath}
    \match_{G\alpha}^{\cat D} \diag X =
    \lim_{\matchcat{\cat D}{G\alpha}} \diag X
    \qquad\text{and}\qquad
    \match_{G\alpha}^{\cat D} \diag Y =
    \lim_{\matchcat{\cat D}{G\alpha}} \diag Y
  \end{displaymath}
  and we define  $P_{G\alpha}^{\cat D}$ by letting the diagram
  \begin{displaymath}
    \xymatrix{
      {P_{G\alpha}^{\cat D}} \ar@{..>}[r] \ar@{..>}[d]
      & {\diag Y_{G\alpha}} \ar[d]\\
      {\match_{G\alpha}^{\cat D} \diag X}
      \ar[r]
      & {\match_{G\alpha}^{\cat D} \diag Y}
    }
  \end{displaymath}
  be a pullback.  The functor $G_{*}\colon \matchcat{\cat C}{\alpha}
  \to \matchcat{\cat D}{G\alpha}$ (see \propref{prop:CatFacts})
  induces natural maps
  \begin{align*}
    \match_{G\alpha}^{\cat D}\diag X =
    \lim_{\matchcat{\cat D}{G\alpha}} \diag X &\longrightarrow
    \lim_{\matchcat{\cat C}{\alpha}} G^{*}\diag X =
    \match_{\alpha}^{\cat C}G^{*}\diag X\\
    \match_{G\alpha}^{\cat D}\diag Y =
    \lim_{\matchcat{\cat D}{G\alpha}} \diag Y &\longrightarrow
    \lim_{\matchcat{\cat C}{\alpha}} G^{*}\diag Y =
    \match_{\alpha}^{\cat C}G^{*}\diag Y
  \end{align*}
  and so we have a map of pullbacks and relative matching maps
  \begin{displaymath}
    \xymatrix@=.8em{
      & {(G^{*}\diag X)_{\alpha}} \ar[dr]\\
      {\diag X_{G\alpha}} \ar@{=}[ur] \ar[dr]
      && {P^{\cat C}_{\alpha}} \ar[rr] \ar'[d][dd]
      && {(G^{*}\diag Y)_{\alpha}} \ar[dd]\\
      & {P^{\cat D}_{G\alpha}} \ar[ur] \ar[rr] \ar[dd]
      && {\diag Y_{G\alpha}} \ar[ur] \ar[dd]\\
      && {\match^{\cat C}_{\alpha}G^{*}\diag X} \ar'[r][rr]
      && {\match^{\cat C}_{\alpha}G^{*}\diag Y}\\
      & {\match^{\cat D}_{G\alpha}\diag X} \ar[ur] \ar[rr]
      && {\match^{\cat D}_{G\alpha}\diag Y} \ar[ur]
    }
  \end{displaymath}
  and our map $(G^{*}\diag X)_{\alpha} \to P_{\alpha}^{\cat C}$ equals
  the composition
  \begin{displaymath}
    (G^{*}\diag X)_{\alpha} =
    \diag X_{G\alpha} \longrightarrow P_{G\alpha}^{\cat D}
    \longrightarrow P_{\alpha}^{\cat C} \Period
  \end{displaymath}
  Since the map $\diag X \to \diag Y$ is a fibration in $\cat M^{\cat
    D}$, the relative matching map $\diag X_{G\alpha} \to
  P_{G\alpha}^{\cat D}$ is a fibration (see \thmref{thm:RFib}), and so
  it is sufficient to show that the natural map
  \begin{equation}
    \label{eq:Matches}
    P_{G\alpha}^{\cat D} \longrightarrow P_{\alpha}^{\cat C}
  \end{equation}
  is a fibration.  That statement is the content of
  \propref{prop:MatchFib} (in \secref{sec:PrfMatchFib}, below) which
  (along with \propref{prop:MatchIso} in \secref{sec:PrfMatchIso})
  will complete the proof of \thmref{thm:RightQuil}.
\end{proof}

\subsection{Statement and proof of \propref{prop:MatchFib}}
\label{sec:PrfMatchFib}

The purpose of this section is to state and prove the following
proposition, which (along with \propref{prop:MatchIso} in
\secref{sec:PrfMatchIso}) will complete the proof of
\thmref{thm:RightQuil}.

\begin{prop}
  \label{prop:MatchFib}
  For every object $\alpha$ of $\cat C$, the map 
  \begin{displaymath}
    P_{G\alpha}^{\cat D} \longrightarrow P_{\alpha}^{\cat C}
  \end{displaymath}
  from \eqref{eq:Matches} is a fibration.
\end{prop}

The proof of \propref{prop:MatchFib} is intricate, but it does not
require any new definitions.  To aid the reader, here is the structure
of the argument:
\begin{equation}
  \label{eq:RightQuilFlow}
  \vcenter{
    \xymatrix@C=1.5em{
      \text{\propref{prop:isoprime}}\ar@{=>}[r]
      & \text{\propref{prop:MatchFib}}
      & \\
      \text{\lemref{lem:reedy}}\ar@{=>}[r]
      & \text{\propref{prop:IndFibr}}\ar@{=>}[u]
      & \text{\lemref{lem:lastfib}}\ar@{=>}[l]\\
      & \txt{\lemref{lem:pbCtoCprime} \& \\
        \diagref{diag:bigcube}}\ar@{=>}[r]
      & \text{\lemref{lem:PBf}}\ar@{=>}[u] 
    }
  }
\end{equation}

We will start with the proof of \propref{prop:MatchFib} and then, as
in the proof of \thmref{thm:RightQuil}, we will work our way backward
from it.

\begin{proof}[Proof of \propref{prop:MatchFib}]
  If the degree of $\alpha$ is $k$, we define a nested sequence of
  subcategories of $\matchcat{\cat D}{G\alpha}$
  \begin{equation}
    \label{eq:MatchFibCats}
    \cat A_{-1} \subset \cat A_{0} \subset \cat A_{1} \subset \cdots
    \subset \cat A_{k-1} =  \matchcat{\cat D}{G\alpha} 
  \end{equation}
  by letting $\cat A_{i}$ for $-1 \le i \le k-1$ be the full
  subcategory of $\matchcat{\cat D}{G\alpha}$ with objects the union
  of
  \begin{itemize}
  \item the objects of $\matchcat{\cat D}{G\alpha}$ whose target is of
    degree at most $i$, and
  \item the image under $G_{*}\colon \matchcat{\cat C}{\alpha} \to
    \matchcat{\cat D}{G\alpha}$ (see \propref{prop:CatFacts}) of the
    objects of $\matchcat{\cat C}{\alpha}$.
  \end{itemize}
  The functor $G_{*}\colon \matchcat{\cat C}{\alpha} \to
  \matchcat{\cat D}{G\alpha}$ factors through $\cat A_{-1}$ and, since
  there are no objects of negative degree, this functor, which by
  abuse of notation we will also call $G_{*}\colon \matchcat{\cat
    C}{\alpha} \to \cat A_{-1}$, maps onto the objects of $\cat A_{-1}$.

  In fact, we claim that the functor $G_{*}\colon \matchcat{\cat
    C}{\alpha} \to \cat A_{-1}$ is left cofinal (see
  \defref{def:cofinal}) and thus induces isomorphisms
  \begin{displaymath}
  \lim_{\cat A_{-1}} \diag X \iso
  \lim_{\matchcat{\cat C}{\alpha}} G^{*}\diag X
  \qquad\text{and}\qquad
  \lim_{\cat A_{-1}}
  \diag Y \iso \lim_{\matchcat{\cat C}{\alpha}} G^{*}\diag Y
  \end{displaymath}
  (see \thmref{thm:CofinalIso}).  To see this, note that every object
  of $\cat A_{-1}$ is of the form $G\sigma\colon G\alpha \to G\beta$
  for some object $\sigma\colon \alpha \to \beta$ of
  $\matchcat{\cat C}{\alpha}$ and \propref{prop:CatFacts} implies that
  the overcategory
  $\bovercat{G_{*}}{(G\sigma\colon G\alpha \to G\beta)}$ is exactly
  the category of inverse $\cat C$-factorizations of
  $(\alpha, G\sigma)$, and so (since $G$ is a fibering Reedy functor)
  its nerve must be either empty or connected.  Since it is not empty
  (it contains the vertex
  $(\alpha \xrightarrow{\sigma} \beta, 1_{G\beta})$), it is connected,
  and so $G_{*}\colon \matchcat{\cat C}{\alpha} \to \cat A_{-1}$ is
  left cofinal.

  The sequence of inclusions of categories \eqref{eq:MatchFibCats}
  thus induces sequences of maps
  \begin{displaymath}
    \begin{gathered}
      \lim_{\matchcat{\cat D}{G\alpha}} \diag X = \lim_{\cat A_{k-1}}
      \diag X \to \lim_{\cat A_{k-2}} \diag X \to \cdots \to
      \lim_{\cat A_{0}} \diag X \to \lim_{\cat A_{-1}} \diag X \iso
      \lim_{\matchcat{\cat C}{\alpha}} G^{*}\diag X\\
      \lim_{\matchcat{\cat D}{G\alpha}} \diag Y = \lim_{\cat A_{k-1}}
      \diag Y \to \lim_{\cat A_{k-2}} \diag Y \to \cdots \to
      \lim_{\cat A_{0}} \diag Y \to \lim_{\cat A_{-1}} \diag Y \iso
      \lim_{\matchcat{\cat C}{\alpha}} G^{*}\diag Y\Period
    \end{gathered}
  \end{displaymath}
 For $-1 \le i \le k-1$ we let $P_{i}$ be the pullback
  \begin{displaymath}
    \xymatrix{
      {P_{i}} \ar@{..>}[r] \ar@{..>}[d]
      & {\diag Y_{G\alpha}} \ar[d]\\
      {\lim_{\cat A_{i}}\diag X} \ar[r]
      & {\lim_{\cat A_{i}}\diag Y \Period}
    }
  \end{displaymath}
  Since we have an evident map of diagrams
  \begin{displaymath}
    \Big(\lim_{\cat A_{i+1}}\diag X\to \lim_{\cat A_{i+1}}\diag
    Y\leftarrow \diag Y_{G\alpha}  \Big) \longrightarrow
    \Big(\lim_{\cat A_{i}}\diag X\to \lim_{\cat A_{i}}\diag
    Y\leftarrow \diag Y_{G\alpha}  \Big)
  \end{displaymath}
  we also get an induced map $P_{i+1}\to P_i$ of pullbacks.
  We thus have a factorization of
  \eqref{eq:Matches} as
  \begin{displaymath}
    P_{G\alpha}^{\cat D} = P_{k-1} \longrightarrow P_{k-2}
    \longrightarrow \cdots \longrightarrow
    P_{-1} \iso P_{\alpha}^{\cat C} \Comma
  \end{displaymath}
  and we will show that the map $P_{i+1} \to P_{i}$ is a fibration for
  $-1 \le i \le k-2$.

  The objects of $\cat A_{i+1}$ that are not in $\cat A_{i}$ are maps
  $G\alpha \to \beta$ where $\beta$ is of degree $i+1$, and this set
  of maps can be divided into two subsets:
  \begin{itemize}
  \item the set $S_{i+1}$ of maps $G\alpha\to\beta$ for which the
    category of inverse $\cat C$-factorizations of $(\alpha, G\alpha
    \to \beta)$ is nonempty, and
  \item the set $T_{i+1}$ of maps for which the category of inverse
    $\cat C$-factorizations of $(\alpha, G\alpha \to \beta)$ is empty.
  \end{itemize}
  We let $\cat A'_{i+1}$ be the full subcategory of $\matchcat{\cat
    D}{G\alpha}$ with objects the union of $S_{i+1}$ with the objects
  of $\cat A_{i}$, and define $P'_{i+1}$ as the pullback
  \begin{displaymath}
    \xymatrix{
      {P'_{i+1}} \ar@{..>}[r] \ar@{..>}[d]
      & {\diag Y_{G\alpha}} \ar[d]\\
      {\lim_{\cat A'_{i+1}}\diag X} \ar[r]
      & {\lim_{\cat A'_{i+1}}\diag Y \Period}
    }
  \end{displaymath}
  We have inclusions of categories $\cat A_{i} \subset \cat A'_{i+1}
  \subset \cat A_{i+1}$, and the maps
  \begin{displaymath}
    \lim_{\cat A_{i+1}} \diag X \longrightarrow
    \lim_{\cat A_{i}}\diag X
    \qquad\text{and}\qquad
    \lim_{\cat A_{i+1}} \diag Y \longrightarrow
    \lim_{\cat A_{i}}\diag Y
  \end{displaymath}
  factor as
  \begin{displaymath}
    \lim_{\cat A_{i+1}} \diag X \longrightarrow
    \lim_{\cat A'_{i+1}} \diag X \longrightarrow
    \lim_{\cat A_{i}}\diag X
    \qquad\text{and}\qquad
    \lim_{\cat A_{i+1}} \diag Y \longrightarrow
    \lim_{\cat A'_{i+1}} \diag Y \longrightarrow
    \lim_{\cat A_{i}}\diag Y \Period
  \end{displaymath}
  These factorizations induce a factorization
  \begin{equation}
    \label{eq:factorization}
    P_{i+1} \longrightarrow  P'_{i+1} \longrightarrow P_{i}
  \end{equation}
  of the map $P_{i+1} \to P_{i}$, and we have the commutative diagram
  \begin{displaymath}
    \xymatrix@=.8em{
      && {P_{i}} \ar[rrr] \ar'[d]'[dd][ddd]
      &&& {\diag Y_{G\alpha}} \ar[ddd]\\
      & {P'_{i+1}} \ar[ur] \ar[rrr] \ar'[d][ddd]
      &&& {\diag Y_{G\alpha}} \ar@{=}[ur] \ar[ddd]\\
      {P_{i+1}} \ar[ur] \ar[rrr] \ar[ddd]
      &&& {\diag Y_{G\alpha}} \ar@{=}[ur] \ar[ddd]\\
      && {\lim_{\cat A_{i}}\diag X} \ar'[r]'[rr][rrr]
      &&& {\lim_{\cat A_{i}}\diag Y}\\
      & {\lim_{\cat A'_{i+1}}\diag X} \ar[ur] \ar'[rr][rrr]
      &&& {\lim_{\cat A'_{i+1}}\diag Y} \ar[ur]\\
      {\lim_{\cat A_{i+1}}\diag X} \ar[ur] \ar[rrr]
      &&& {\lim_{\cat A_{i+1}}\diag Y} \ar[ur]
    }
  \end{displaymath}
  \propref{prop:isoprime} below asserts that the map $P'_{i+1} \to
  P_{i}$ is an isomorphism and \propref{prop:IndFibr} asserts that the
  map $P_{i+1} \to P'_{i+1}$ is a fibration.  Hence, the map
  $P_{G\alpha}^{\cat D}\to P_{\alpha}^{\cat C}$ is a fibration as
  well.
\end{proof}

\begin{prop}
  \label{prop:isoprime}
  For $-1 \le i \le k-2$, the map $P'_{i+1} \to P_{i}$ in
  \eqref{eq:factorization} is an isomorphism.
\end{prop}

\begin{proof}
  We will show that for every element $\sigma\colon G\alpha \to \beta$
  of $\cat A_{i}$ the overcategory $\overcat{\cat A_{i}}{\sigma}$ is
  nonempty and connected, which will imply that the inclusion
  $\cat A_{i} \subset \cat A'_{i+1}$ is left cofinal (see
  \defref{def:cofinal}).  This will imply that the maps
  $\lim_{\cat A'_{i+1}} \diag X \to \lim_{\cat A_{i}}\diag X$ and
  $\lim_{\cat A'_{i+1}} \diag Y \to \lim_{\cat A_{i}}\diag Y$ are
  isomorphisms (see \thmref{thm:CofinalIso}), and so the induced map
  $P'_{i+1} \to P_{i}$ is an isomorphism.

  If $\sigma\colon G\alpha \to \beta$ is an element of $\cat A_{i}$,
  then the overcategory $\overcat{\cat A_{i}}{\sigma}$ has the
  terminal object $1_{\sigma}$ and is thus nonempty and connected.

  Now suppose that $\sigma\colon G\alpha \to \beta$ is an object of
  $\cat A'_{i+1}$ that is not in $\cat A_{i}$.  The objects of
  $\overcat{\cat A_{i}}{\sigma}$ are commutative diagrams
  \begin{displaymath}
    \xymatrix@=.6em{
      & {G\alpha} \ar[dl]_{\nu} \ar[dr]^{\sigma}\\
      {\gamma} \ar[rr]_{\mu}
      && {\beta}
    }
  \end{displaymath}
  where $\nu\colon G\alpha \to \gamma$ is in $\cat A_{i}$ and $\mu$ is
  in $\inv{\cat D}$.  Since $\beta$ is of degree $i+1$ and $\mu$
  lowers degree (because $\mu$ cannot be an identity map, since
  $\sigma$ isn't in $\cat A_{i}$), the degree of $\gamma$ must be
  greater than $i+1$, and so the map $\nu\colon G\alpha \to \gamma$
  must be of the form $G\nu'\colon G\alpha \to G\gamma'$ for some map
  $\nu'\colon \alpha \to \gamma'$ in $\matchcat{\cat C}{\alpha}$.
  Thus, the objects of $\overcat{\cat A_{i}}{\sigma}$ are pairs
  $\bigl((\nu'\colon \alpha\to \gamma'), (\mu\colon G\gamma' \to
  \beta)\bigr)$ where $\nu'\colon \alpha \to \gamma'$ is a
  non-identity map of $\inv{\cat C}$, $\mu\colon G\gamma' \to \beta$
  is in $\inv{\cat D}$, and $\mu\circ G\nu' = \sigma$, and
  $\overcat{\cat A_{i}}{\sigma}$ is the category of inverse
  $\cat C$-factorizations of $(\alpha, \sigma)$ (see
  \propref{prop:CatFacts}).  Since $G$ is a fibering Reedy functor,
  the nerve of the category of inverse $\cat C$-factorizations of
  $(\alpha, \sigma)$ is either empty or connected.  Since it is
  nonempty (because $\sigma\colon G\alpha \to \beta$ is an element of
  $S_{i+1}$), the nerve of the overcategory
  $\overcat{\cat A_{i}}{\sigma}$ is nonempty and connected.
\end{proof}

\begin{prop}
  \label{prop:IndFibr}
  For $-1 \le i \le k-2$, the map $P_{i+1} \to P'_{i+1}$ in
  \eqref{eq:factorization} is a fibration.
\end{prop}

The proof of \propref{prop:IndFibr} is more intricate; the reader
might wish to refer to the chart \eqref{eq:RightQuilFlow} for its
structure.  Before we can present it, we will need several lemmas.
For the first one, the reader should recall the definition of the sets
$T_i$ from the proof of \propref{prop:MatchFib}.

\begin{lem}
  \label{lem:pbCtoCprime}
  For every $\cat D$-diagram $\diag Z$ in $\cat M$ there is a natural
  pullback square
  \begin{equation}
    \label{diag:pbCtoCprime}
    \vcenter{
      \xymatrix{
        {\lim_{\cat A_{i+1}} \diag Z} \ar[r] \ar[d]
        & {\lim_{\cat A'_{i+1}} \diag Z} \ar[d]\\
        {\prod_{(G\alpha\to\beta)\in T_{i+1}}
          \hspace{-1.5em}\diag Z_{\beta}} \ar[r]
        & {\prod_{(G\alpha\to\beta)\in T_{i+1}}
          \lim_{\matchcat{\cat D}{\beta}} \diag Z \Period}
      }
    }
  \end{equation}
\end{lem}

\begin{proof}
  For every element $\sigma\colon G\alpha\to\beta$ of $T_{i+1}$, every
  object of the matching category $\matchcat{\cat D}{\beta}$ is a map
  to an object of degree at most $i$, and so we have a functor
  $\matchcat{\cat D}{\beta} \to \cat A'_{i+1}$ that takes $\beta \to
  \gamma$ to the composition $G\alpha\xrightarrow{\sigma}
  \beta\to\gamma$; this induces the map $\lim_{\cat A'_{i+1}} \diag Z
  \longrightarrow \lim_{\matchcat{\cat D}{\beta}} \diag Z$ that is the
  projection of the right hand vertical map onto the factor indexed by
  $\sigma$.  We thus have a commutative square as in
  \diagref{diag:pbCtoCprime}.

  The objects of $\cat A_{i+1}$ are the objects of $\cat A'_{i+1}$
  together with the elements of $T_{i+1}$, and so a map to
  $\lim_{\cat A_{i+1}} \diag Z$ is determined by its postcompositions
  with the above maps to $\lim_{\cat A'_{i+1}} \diag Z$ and
  $\prod_{(G\alpha\to\beta)\in T_{i+1}} \diag Z_{\beta}$.  Since there
  are no non-identity maps in $\cat A_{i+1}$ with codomain an element
  of $T_{i+1}$ (because
  $\invfact{\cat C}{\alpha}{G\alpha \to \beta} = \emptyset$), and the
  only non-identity maps with domain an element $G\alpha\to\beta$ of
  $T_{i+1}$ are the objects of the matching category
  $\matchcat{\cat D}{\beta}$, maps to $\lim_{\cat A'_{i+1}} \diag Z$
  and to $\prod_{(G\alpha\to\beta)\in T_{i+1}} X_{\beta}$ determine a
  map to $\lim_{\cat A_{i+1}} \diag Z$ if and only if their
  compositions to
  $\prod_{(G\alpha\to\beta)\in T_{i+1}} \lim_{\matchcat{\cat
      D}{\beta}} \diag Z$ agree.  Thus, the diagram is a pullback
  square.
\end{proof}

Now define $Q$ and $R$ by letting the squares
\begin{equation}
  \label{diag:pullbacks}
  \vcenter{
    \xymatrix{
      {Q} \ar@{..>}[r] \ar@{..>}[d]
      & {\lim_{\cat A'_{i+1}} \diag X} \ar[d]\\
      {\lim_{\cat A_{i+1}} \diag Y} \ar[r]
      & {\lim_{\cat A'_{i+1}} \diag Y}
    }
  }
  \qquad\text{and}\quad
  \vcenter{
    \xymatrix@=1.7em{
      {R} \ar@{..>}[r] \ar@{..>}[d]
      & {\prod_{(G\alpha\to\beta)\in T_{i+1}}
        \lim_{\matchcat{\cat D}{\beta}} \diag X} \ar[d]\\
      {\prod_{(G\alpha\to\beta)\in T_{i+1}}
        \hspace{-1.5em}\diag Y_{\beta}} \ar[r]
      & {\prod_{(G\alpha\to\beta)\in T_{i+1}}
        \lim_{\matchcat{\cat D}{\beta}} \diag Y}
    }
  }
\end{equation}
be pullbacks, and consider the commutative diagram
\begin{equation}
  \label{diag:bigcube}
  \vcenter{
    \xymatrix@=.8em{
      {\lim_{\cat A_{i+1}} \diag X} \ar[rrr]^{s} \ar[drr]^(.7){a}
      \ar[ddr]_{\delta} \ar[ddd]_{u}
      &&& {\lim_{\cat A'_{i+1}} \diag X} \ar[ddr]^{\beta}
      \ar'[dd][ddd]^{v}\\
      && {Q} \ar[ur]_{c} \ar[dl]_{d} \ar'[d][ddd]^(.3){g}\\
      & {\lim_{\cat A_{i+1}}\diag Y} \ar[rrr]^(.7){s'}
      \ar[ddd]_(.75){u'}
      &&& {\lim_{\cat A'_{i+1}}\diag Y} \ar[ddd]^{v'}\\
      {\hspace{-1em}\prod_{(G\alpha\to\beta)\in T_{i+1}}
        \hspace{-1.5em}\diag X_{\beta}} \ar[ddr]_{\gamma}
      \ar[drr]|!{[ru];[rdd]}{\hole}^(.75){b}
      \ar[rrr]|!{[ruu];[rd]}{\hole}^{t}
      |!{[uurr];[rrd]}{\hole}
      &&& {\prod_{\substack{\phantom{X}\\(G\alpha\to\beta)\in T_{i+1}}}
        \hspace{-1.5em}\lim_{\matchcat{\cat D}{\beta}}\diag X} \ar[ddr]\\
      && {R} \ar[ur]_-{e}  \ar[dl]^(.3){f}\\
      & {\hspace{-1.5em}\prod_{(G\alpha\to\beta)\in T_{i+1}}
        \hspace{-1.5em}\diag Y_{\beta}} \ar[rrr]_-{t'}
      &&& {\hspace{-1.5em}
        \prod_{\substack{\phantom{X}\\(G\alpha\to\beta)\in T_{i+1}}}
        \hspace{-1.5em}\lim_{\matchcat{\cat D}{\beta}}\diag Y}
    }
  }
\end{equation}
\lemref{lem:pbCtoCprime} implies that the front and back rectangles
are pullbacks.

\begin{lem}
  \label{lem:PBf}
  The square
  \begin{equation}
    \label{diag:PBf}
    \vcenter{
      \xymatrix{
        {\lim_{\cat A_{i+1}}\diag X} \ar[r]^-{a} \ar[d]_{u}
        & {Q} \ar[d]^{g}\\
        {\prod_{(G\alpha\to\beta)\in T_{i+1}}
          \hspace{-1.5em}\diag X_{\beta}} \ar[r]_-{b}
        & {R}
      }
    }
  \end{equation}
  is a pullback.
\end{lem}

\begin{proof}
  Let $W$ be an object of $\cat M$ and let $h\colon W \to
  \prod_{(G\alpha\to\beta)\in T_{i+1}} \diag X_{\beta}$ and $k\colon W
  \to Q$ be maps such that $gk = bh$; we will show that there is a
  unique map $\phi\colon W \to \lim_{\cat A_{i+1}} \diag X$ such that
  $a\phi = k$ and $u\phi = h$.
  \begin{displaymath}
    \xymatrix{
      {W} \ar@/^3ex/[drr]^{k} \ar@/_3ex/[ddr]_{h}
      \ar@{..>}[dr]^{\phi}\\
      &{\lim_{\cat A_{i+1}} \diag X} \ar[r]^-{a} \ar[d]_{u}
      & {Q} \ar[d]^{g}\\
      &{\prod_{(G\alpha\to\beta)\in T_{i+1}}
        \hspace{-1.5em}\diag X_{\beta}} \ar[r]_-{b}
      & {R}
    }
  \end{displaymath}
  The map $ck\colon W \to \lim_{\cat A'_{i+1}} \diag X$ has the
  property that $v(ck) = egk = ebh = th$, and since the back rectangle
  of \diagref{diag:bigcube} is a pullback, the maps $ck$ and $h$
  induce a map $\phi\colon W \to \lim_{\cat A_{i+1}} \diag X$ such
  that $u\phi = h$ and $s\phi = ck$.  We must show that $a\phi = k$,
  and since $Q$ is a pullback as in \diagref{diag:pullbacks}, this is
  equivalent to showing that $ca\phi = ck$ and $da\phi = dk$.

  Since $ck = s\phi = ca\phi$, we need only show that $da\phi = dk$.
  Since the front rectangle of \diagref{diag:bigcube} is a pullback,
  it is sufficient to show that $s'da\phi = s'dk$ and $u'da\phi =
  u'dk$.  For the first of those, we have
  \begin{displaymath}
    s'da\phi = s'\delta\phi = \beta s\phi = \beta ck = s'dk\\
  \end{displaymath}
  and for the second, we have
  \begin{displaymath}
    u'da\phi = u'\delta\phi = \gamma u\phi = fbu\phi = fbh = fgk =
    u'dk \Period
  \end{displaymath}
  Thus, the map $\phi$ satisfies $a\phi = k$.

  To see that $\phi$ is the unique such map, let $\psi\colon W \to
  \lim_{\cat A_{i+1}} \diag X$ be another map such that $a\psi = k$
  and $u\psi = h$.  We will show that $s\psi = s\phi$ and $u\psi =
  u\phi$; since the back rectangle of \diagref{diag:bigcube} is a
  pullback, this will imply that $\psi = \phi$.

  Since $u\psi = h = u\phi$, we need only show that $s\psi = s\phi$,
  which follows because $s\psi = ca\psi = ck = s\phi$.
\end{proof}

\begin{lem}
  \label{lem:lastfib}
  If $\diag X \to \diag Y$ is a fibration of $\cat D$-diagrams, then
  the natural map
  \begin{displaymath}
    \lim_{\cat A_{i+1}} \diag X \longrightarrow
    Q = \pullback{\lim_{\cat A'_{i+1}}\diag X}
    {(\lim_{\cat A'_{i+1}}\diag Y)}{\lim_{\cat A_{i+1}}\diag Y}
  \end{displaymath}
  is a fibration.
\end{lem}

\begin{proof}
  \lemref{lem:PBf} gives us the pullback square in \diagref{diag:PBf}
  where $Q$ and $R$ are defined by the pullbacks in
  \diagref{diag:pullbacks}.  Since $\diag X \to \diag Y$ is a
  fibration of $\cat D$-diagrams, the map
  $\prod_{(G\alpha\to\beta)\in T_{i+1}} \diag X_{\beta} \to R$ is a
  product of fibrations (see \defref{def:Matchobj}~(6)) and is thus a
  fibration, and so the map
  $\lim_{\cat A_{i+1}}\diag X \to Q = \pullback{\lim_{\cat
      A'_{i+1}}\diag X} {(\lim_{\cat A'_{i+1}}\diag Y)}{\lim_{\cat
      A_{i+1}}\diag Y}$ is a pullback of a fibration and is thus a
  fibration.
\end{proof}

\begin{lem}[Reedy]
  \label{lem:reedy}
  If both the front and back squares in the diagram
  \begin{displaymath}
    \xymatrix@C=5ex@R=1em{
      {A} \ar[rr] \ar[dr]_{f_{A}} \ar[dd]
      && {B} \ar[dr]^{f_{B}} \ar'[d][dd]\\
      & {A'} \ar[rr] \ar[dd]
      && {B'} \ar[dd]\\
      {C} \ar[dr]_{f_{C}} \ar'[r][rr]
      && {D} \ar[dr]^{f_{D}}\\
      & {C'} \ar[rr]
      && {D'}
    }
  \end{displaymath}
  are pullbacks and both $f_{B}\colon B \to B'$ and $C \to
  \pullback{C'}{D'}{D}$ are fibrations, then $f_{A}\colon A \to A'$ is
  a fibration.
\end{lem}

\begin{proof}
  This is the dual of a lemma of Reedy (see
  \cite{MCATL}*{Lem.~7.2.15 and Rem.~7.1.10}).
\end{proof}

\begin{proof}[Proof of \propref{prop:IndFibr}]
  We have a commutative diagram
  \begin{displaymath}
    \xymatrix@=.25em{
      {P_{i+1}} \ar[rr] \ar[dr] \ar[dd]
      && {Y_{G\alpha}} \ar[dr] \ar'[d][dd]\\
      & {P'_{i+1}} \ar[rr] \ar[dd]
      && {Y_{G\alpha}} \ar[dd]\\
      {\lim_{\cat A_{i+1}}\diag X} \ar[dr] \ar'[r][rr]
      && {\lim_{\cat A_{i+1}}\diag Y} \ar[dr]\\
      & {\lim_{\cat A'_{i+1}}\diag X} \ar[rr]
      && {\lim_{\cat A'_{i+1}}\diag Y}
    }
  \end{displaymath}
  in which the front and back squares are pullbacks (by definition), and so
  \lemref{lem:reedy} implies that it is sufficient to show that the
  map
  \begin{displaymath}
    \lim_{\cat A_{i+1}} \diag X \longrightarrow
    \pullback{\lim_{\cat A'_{i+1}}\diag X}
    {(\lim_{\cat A'_{i+1}}\diag Y)}{\lim_{\cat A_{i+1}}\diag Y}
  \end{displaymath}
  is a fibration; that is the statement of \lemref{lem:lastfib}.
\end{proof}

\subsection{Statement and proof of \propref{prop:MatchIso}}
\label{sec:PrfMatchIso}

The purpose of this section is to state and prove the following
proposition, which (along with \propref{prop:MatchFib} in
\secref{sec:PrfMatchFib}) completes the proof of
\thmref{thm:RightQuil}.

\begin{prop}
  \label{prop:MatchIso}
  Let $G\colon \cat C \to \cat D$ be a fibering Reedy functor and let
  $\diag X$ be a $\cat D$-diagram in a model category $\cat M$.  If
  $\alpha$ is an object of $\cat C$ for which there is an object
  $\alpha \to \gamma$ of $\matchcat{\cat C}{\alpha}$ (i.e., a
  non-identity map $\alpha \to \gamma$ in $\inv{\cat C}$) that $G$
  takes to an identity map in $\inv{\cat D}$, then the matching map
  $(G^{*}\diag X)_{\alpha} \to \match^{\cat C}_{\alpha}(G^{*}\diag X)$
  of $G^{*}\diag X$ (see \defref{def:InducedDiag}) at $\alpha$ is an
  isomorphism.
\end{prop}

The  proof will require several preliminary definitions
and results.
\begin{defn}
  \label{def:Gkernel}
  The \emph{$G$-kernel at $\alpha$} is the full subcategory of the
  matching category $\matchcat{\cat C}{\alpha}$ with objects the
  non-identity maps $\alpha \to \gamma$ in $\inv{\cat C}$ that $G$
  takes to the identity map of $G\alpha$.
\end{defn}

If $\alpha \to \gamma$ is an object of the $G$-kernel at $\alpha$,
then the map $(G^{*}\diag X)_{\alpha} \to (G^{*}\diag X)_{\gamma}$ is
the identity map.  Note that the $G$-kernel at $\alpha$ is not usually
left cofinal in $\matchcat{\alpha}{\cat C}$.

\begin{lem}
  \label{lem:IdSubcat}
  Under the hypotheses of \propref{prop:MatchIso}, the nerve of the
  $G$-kernel at $\alpha$ is connected.
\end{lem}

\begin{proof}
  Since $G$ is a fibering Reedy functor, the nerve of the category
  $\invfact{\cat C}{\alpha}{1_{G\alpha}}$ of inverse $\cat
  C$-factorizations of $(\alpha, 1_{G\alpha})$ is connected, and there
  is an isomorphism from the $G$-kernel at $\alpha$ to $\invfact{\cat
    C}{\alpha}{1_{G\alpha}}$ that takes the object $\alpha \to \gamma$
  to the object $\bigl((\alpha \to \gamma), (1_{G\alpha})\bigr)$.
\end{proof}

The matching object $\match^{\cat C}_{\alpha}(G^{*}\diag X)$ is the
limit of a $\matchcat{\cat C}{\alpha}$-diagram (which we will also
denote by $G^{*}\diag X$); we will refer to that diagram as the
\emph{matching diagram}.  The restriction of the matching diagram to
the $G$-kernel at $\alpha$ is a diagram in which every object goes to
$\diag X_{G\alpha} = (G^{*}\diag X)_{\alpha}$ and every map goes to
the identity map of $\diag X_{G\alpha}$, because if there is a
commutative triangle
\begin{displaymath}
    \xymatrix@=.8em{
      &{\alpha} \ar[dl]_{f} \ar[dr]^{f'}\\
      {\gamma} \ar[rr]_{\tau}
      && {\gamma'}
    }
\end{displaymath}
in $\inv{\cat C}$ in which $Gf = Gf' = 1_{G\alpha}$, then $G\tau \circ
1_{G\alpha} = 1_{G\alpha}$, and so $G\tau = 1_{G\alpha}$.  Together
with \lemref{lem:IdSubcat}, this implies the following.
\begin{lem}
  \label{lem:IdDiag}
  Under the hypotheses of \propref{prop:MatchIso}, the restriction
  of the matching diagram to the $G$-kernel at $\alpha$ is a connected
  diagram in which every object goes to $\diag X_{G\alpha}$ and every
  map goes to the identity map of $\diag X_{G\alpha}$.
\end{lem}

We will prove \propref{prop:MatchIso} by showing that for every object
$W$ of $\cat M$ the matching map induces an isomorphism of sets of
maps
\begin{equation}
  \label{eq:MatchMap}
  \cat M \bigl(W, (G^{*}\diag X)_{\alpha}\bigr) \longrightarrow
  \cat M\bigl(W,\match^{\cat C}_{\alpha}(G^{*}\diag X)\bigr)
\end{equation}
(see \propref{prop:DetectIso}).  The matching object $\match^{\cat
  C}_{\alpha}(G^{*}\diag X)$ is the limit of the matching diagram, and
so maps from $W$ to $\match^{\cat C}_{\alpha}(G^{*}\diag X)$
correspond to maps from $W$ to the matching diagram.
\lemref{lem:IdDiag} implies that if we restrict the matching diagram
to the $G$-kernel at $\alpha$, then maps from $W$ to the restriction
of that diagram to the $G$-kernel at $\alpha$ correspond to maps from
$W$ to $(G^{*}\diag X)_{\alpha}$, and that fact allows us to define a
potential inverse to \eqref{eq:MatchMap}.  All that remains is to show
that our potential inverse is actually an inverse.

If $\alpha \to \beta$ and $\alpha \to \gamma$ are objects of the
matching category and there is a map $\tau\colon (\alpha \to \beta)
\to (\alpha \to \gamma)$ in the matching category, i.e., a commutative
diagram
\begin{displaymath}
  \xymatrix@=.6em{
    & {\alpha} \ar[dl] \ar[dr]\\
    {\beta} \ar[rr]_{\tau}
    && {\gamma \Comma}
  }
\end{displaymath}
then for every object $W$ of $\cat M$ and map from $W$ to the matching
diagram, the projection of that map onto $(\alpha \to \gamma)$ is
entirely determined by its projection onto $(\alpha\to\beta)$; we will
describe this by saying that the object $(\alpha\to\gamma)$ is
\emph{controlled} by the object $(\alpha\to\beta)$.  Similarly, if
there is a commutative triangle
\begin{displaymath}
  \xymatrix@=.6em{
    & {\alpha} \ar[dl] \ar[dr]\\
    {\gamma} \ar[rr]_{\tau}
    && {\gamma'}
  }
\end{displaymath}
in the matching category such that $G\tau$ is an identity map, then we
will say that the object $(\alpha\to\gamma)$ is \emph{controlled} by
the object $(\alpha\to\gamma')$ \emph{and} that the object
$(\alpha\to\gamma')$ is \emph{controlled} by the object $(\alpha\to
\gamma)$.  We will show by a downward induction on degree that all
objects of the matching category are controlled by objects of the
$G$-kernel at $\alpha$ (see \defref{def:controlled} and
\propref{prop:AllControlled}).

\begin{defn}
  \label{def:Gequiv}
  We define an equivalence relation on the set of objects of
  $\matchcat{\cat C}{\alpha}$, called \emph{$G$-equivalence at
    $\alpha$}, as the equivalence relation generated by the relation
  under which $f\colon \alpha \to \gamma$ is equivalent to $f'\colon
  \alpha \to \gamma'$ if there is a commutative triangle
  \begin{displaymath}
    \xymatrix@=.8em{
      &{\alpha} \ar[dl]_{f} \ar[dr]^{f'}\\
      {\gamma} \ar[rr]_{\tau}
      && {\gamma'}
    }
  \end{displaymath}
  with $G\tau$ an identity map.
\end{defn}
If $f$ and $f'$ are $G$-equivalent at $\alpha$, then $Gf = Gf'$, and
there is a zig-zag of identity maps connecting $\diag X_{f}$ and
$\diag X_{f'}$ in the matching diagram.

\begin{defn}
  \label{def:controlled}
  We define the set of \emph{controlled objects} $\{\alpha \to
  \gamma\}$ of the matching category $\matchcat{\cat C}{\alpha}$ by a
  decreasing induction on $\degree(G\gamma)$:
  \begin{enumerate}
  \item If $\alpha \to \gamma$ is an object of
    $\matchcat{\cat C}{\alpha}$ such that
    $\degree(G\gamma) = \degree(G\alpha)$ (i.e., if
    $G(\alpha\to\gamma) = 1_{G\alpha}$), then $\alpha \to \gamma$ is
    controlled.  (That is, all objects of the $G$-kernel at $\alpha$
    are controlled.)  Note that this initial step is non-empty, since
    we have assumed that the $G$-kernel at $\alpha$ is non-empty.
  \item If $0 \le n < \degree(G\alpha)$ and we have defined the
    controlled objects $\alpha \to \delta$ for $n < \degree(\delta)
    \le \degree(G\alpha)$, then we define an object $\alpha \to
    \gamma$ with $\degree(G\gamma) = n$ to be controlled if it is
    $G$-equivalent at $\alpha$ to an object $\alpha \to \gamma'$ that
    has a factorization $\alpha \to \delta \to \gamma'$ in $\inv{\cat
      C}$ such that $\alpha \to \delta$ is an object of
    $\matchcat{\cat C}{\alpha}$ that is controlled.
  \end{enumerate}
\end{defn}

\begin{ex}
  Let $G\colon \cat C \to \cat D$ be the fibering Reedy functor
  between Reedy categories as in the following diagram:
  \begin{displaymath}
    \xymatrix@R=1.5ex@C=1em{
      & {\cat C}  \ar[]+<4.5em,0ex>;[rrrrrr]+<-2.5em,0ex>^{G}
      &&&&&& {\cat D}\\
      & {\alpha} \ar[dr] \ar[dd] \ar[lddd]_{\sigma}
      &&&&&& {a} \ar[dddd]^{f}\\
      && {\beta} \ar[dl]\\
      & {\gamma} \ar[dd]^{\tau}\\
      {\delta} \ar[dr]_{\mu}\\
      & {\epsilon}
      &&&&&& {b}
    }
  \end{displaymath}
  where
  \begin{itemize}
  \item $\cat C$ has five objects, $\alpha$, $\beta$, $\gamma$,
    $\delta$, and $\epsilon$ of degrees $4$, $3$, $2$, $1$, and $0$,
    respectively, and the diagram commutes;
  \item $\cat D$ has two objects, $a$ and $b$ of degrees $1$ and $0$,
    respectively;
  \item $G\alpha = G\beta = G\gamma = a$ and $G$ takes the maps
    between them to $1_{a}$;
  \item $G\delta = G\epsilon = b$ and $G\mu = 1_{b}$; and
  \item $G\sigma = G\tau = f$.
  \end{itemize}
  Every object of $\matchcat{\cat C}{\alpha}$ is controlled:
  \begin{itemize}
  \item The objects $\alpha \to \beta$ and $\alpha \to\gamma$ are
    controlled because of the first part of \defref{def:controlled}.
  \item The object $\alpha \to \epsilon$ is controlled because it is
    $G$-equivalent at $\alpha$ to itself and it factors as $\alpha
    \to\gamma \to\epsilon$ with the object $\alpha \to \gamma$
    controlled.
  \item The object $\sigma$ is controlled because it is $G$-equivalent
    at $\alpha$ to $\alpha \to \epsilon$ and the latter map factors as
    $\alpha \to \gamma \to \epsilon$ where the object $\alpha \to
    \gamma$ is controlled.
  \end{itemize}
  If $\diag X$ is a $\cat D$-diagram in a model category $\cat M$,
  then the induced $\cat C$-diagram $G^{*}\diag X$ has
  \begin{displaymath}
    (G^{*}\diag X)_{\alpha} = (G^{*}\diag X)_{\beta} =
    (G^{*}\diag X)_{\gamma} = \diag X_{a}
    \qquad\text{and}\qquad
    (G^{*}\diag X)_{\delta} = (G^{*}\diag X)_{\epsilon} =
    \diag X_{b}\Comma
  \end{displaymath}
  and the matching object of $(G^{*}\diag X)$ at $\alpha$ is the limit
  of the diagram
  \begin{displaymath}
    \xymatrix@C=1.5em@R=2ex{
      && {\diag X_{a}} \ar[dl]^{1_{\diag X_{a}}}\\
      & {\diag X_{a}} \ar[dd]^{\diag X_{f}}\\
      {\diag X_{b}} \ar[dr]_{1_{\diag X_{b}}}\\
      & {\diag X_{b} \Semicolon}
    }
  \end{displaymath}
  that limit is isomorphic to $\diag X_{a}$, as guaranteed by
  \propref{prop:MatchIso}.
\end{ex}

The set of controlled objects has the following property.
\begin{lem}
  \label{lem:controlled}
  Under the hypotheses of \propref{prop:MatchIso}, if $W$ is an object
  of $\cat M$ and $h,k\colon W \to \match^{\cat C}_{\alpha}(G^{*}\diag
  X)$ are two maps to the matching object of $G^{*}\diag X$ at
  $\alpha$ whose projections onto at least one object of the
  $G$-kernel at $\alpha$ agree, then their projections onto every
  controlled object agree.
\end{lem}

\begin{proof}
  This follows by a decreasing induction as in
  \defref{def:controlled}, using \lemref{lem:IdDiag} and
  \defref{def:controlled}.
\end{proof}

That every object in the example above was controlled was not an
accident, as shown by the following result.

\begin{prop}
  \label{prop:AllControlled}
  Under the hypotheses of \propref{prop:MatchIso}, every object
  $f\colon \alpha \to \gamma$ of $\matchcat{\cat C}{\alpha}$ is
  controlled.
\end{prop}

\begin{proof}
  We will show this by a decreasing induction on the degree of
  $G\gamma$ in $\cat D$, beginning with $\degree(G\alpha)$.  The
  induction is begun because the objects $f\colon \alpha \to \gamma$
  in $\matchcat{\cat C}{\alpha}$ with $\degree(G\gamma) =
  \degree(G\alpha)$ are exactly the objects of the $G$-kernel at
  $\alpha$, since a map in $\inv{\cat D}$ that does not lower degree
  must be an identity map.

  Suppose now that $0 \le n < \degree(G\alpha)$, that every object
  $\alpha \to \delta$ in $\matchcat{\cat C}{\alpha}$ with
  $\degree(G\delta) > n$ is controlled, and that
  $f\colon \alpha \to \gamma$ is an object of
  $\matchcat{\cat C}{\alpha}$ with $\degree(G\gamma) = n$.  We will
  show that there is a map $\tau\colon \epsilon \to \gamma'$ in
  $\invfact{\cat C}{\alpha}{Gf}$ from an object
  $\bigl((h\colon \alpha \to \epsilon), (G\epsilon \to G\gamma)\bigr)$
  with $\degree(\epsilon) > \degree(\gamma)$ to an object
  $\bigl((f'\colon \alpha \to \gamma'), (1\colon G\gamma' \to G\gamma'
  = G\gamma)\bigr)$ that is $G$-equivalent to $f$.  The induction
  hypothesis will then imply that $h\colon \alpha \to \epsilon$ is
  controlled, and since the composition
  $\alpha \xrightarrow{h} \epsilon \xrightarrow{\tau} \gamma'$ equals
  $f'\colon \alpha \to \gamma'$, this will imply that
  $f\colon \alpha \to \gamma$ is controlled.

  Consider the category $\invfact{\cat C}{\alpha}{Gf}$ of inverse
  $\cat C$-factorizations of $(\alpha, Gf\colon G\alpha \to G\gamma)$.
  We first show that if
  $\bigl((f'\colon \alpha \to \gamma'), (1_{G\gamma})\bigr)$ is an
  object of $\invfact{\cat C}{\alpha}{Gf}$ such that $f'$ is
  $G$-equivalent at $\alpha$ to $f$, and if that object is the domain
  of a map to an object
  $\bigl((h\colon \alpha \to \epsilon), (G\epsilon \to
  G\gamma)\bigr)$, then we must have
  $\degree(G\epsilon) = \degree(G\gamma)$ and that target object must
  actually be of the form
  $\bigl((f''\colon \alpha \to \gamma''), (1_{G\gamma})\bigr)$ where
  $f''$ is also $G$-equivalent at $\alpha$ to $f$.  This is because if
  $\tau\colon \gamma' \to \epsilon$ is a map in $\inv{\cat C}$ such
  that $G\tau$ is \emph{not} an identity map, then
  $\degree(G\epsilon) < \degree(G\gamma') = \degree(G\gamma)$, which
  is not possible because an identity map in a Reedy category cannot
  factor through a degree-lowering map.

  The category $\invfact{\cat C}{\alpha}{Gf}$ contains the object
  $\bigl((f\colon \alpha \to \gamma), (1_{G\gamma})\bigr)$ and, if
  $g\colon \alpha \to \delta$ is an object of the $G$-kernel at
  $\alpha$, then it also contains the object
  $\bigl((g\colon \alpha \to \delta), (Gf\colon G\alpha \to
  G\gamma)\bigr)$.  Since $G$ is a fibering Reedy functor, the nerve
  of the category $\invfact{\cat C}{\alpha}{Gf}$ is connected, and so
  there must be a zig-zag of maps in $\invfact{\cat C}{\alpha}{Gf}$
  connecting those two objects.  Since every map in
  $\invfact{\cat C}{\alpha}{Gf}$ with domain an object
  $\bigl((f'\colon \alpha \to \gamma'), (1_{G\gamma})\bigr)$ (where
  $f'\colon \alpha \to \gamma'$ is $G$-equivalent at $\alpha$ to
  $f\colon \alpha \to \gamma$) can have as a target only another such
  object, and the object
  $\bigl((g\colon \alpha \to \delta), (Gf\colon G\alpha \to
  G\gamma)\bigr)$ (with $g\colon \alpha \to \delta$ an object of the
  $G$-kernel at $\alpha$) is at the left end of the zig-zag, the
  zig-zag must look like the following:
  \begin{displaymath}
    \xymatrix@C=0.5em@R=2ex{
      & {\bullet} \ar[dl] \ar[dr]
      && {\left(\substack{(h\colon \alpha \to \epsilon),\\
            (G\epsilon \to G\gamma)}\right)} \ar[dl] \ar[dr]^{\tau}
      && {\bullet} \ar[dl]_-{\sim} \ar[dr]^-{\sim}\\
      {\left(\substack{(g\colon \alpha \to \delta),\\
            (Gf\colon G\alpha\to G\gamma)}\right)}
      && {\hspace{.25in}\bullet\hspace{.25in}}
      && {\left(\substack{(f'\colon \alpha \to \gamma'),\\
            (1_{G\gamma})}\right)}
      && {\left(\substack{(f\colon \alpha \to \gamma),\\(1_{G\gamma})}\right)}
    }
  \end{displaymath}
  That is, the rightmost few maps in the zig-zag can be maps that $G$
  takes to $1_{G\gamma}$ (labelled with ``$\sim$'' in the diagram),
  but at some point in the zig-zag there must be a map going to the
  right, from an object
  $\bigl((h\colon \alpha \to \epsilon), (G\epsilon \to G\gamma)\bigr)$
  of $\invfact{\cat C}{\alpha}{Gf}$ with $h$ \emph{not} $G$-equivalent
  at $\alpha$ to $f$ and a map $\tau\colon \epsilon \to \gamma'$ from
  that object to an object
  $\bigl((f'\colon \alpha \to \gamma'), (1_{G\gamma})\bigr)$ where
  $f'\colon \alpha \to \gamma'$ is $G$-equivalent at $\alpha$ to $f$.

  If we had $\degree(G\epsilon) = \degree(G\gamma)$, then
  $G\tau$ would be an identity map (and so $h$ would be
  $G$-equivalent to $f$) because there would be a commutative triangle
  \begin{displaymath}
    \xymatrix@=.8em{
      {G\epsilon} \ar[rr]^{G\tau} \ar[dr]
      && {G\gamma'} \ar[dl]^{1_{G\gamma'}}\\
      & {G\gamma'}
    }
  \end{displaymath}
  in which the map $G\epsilon \to G\gamma'$ is a map of $\inv{\cat D}$
  that does not lower degree and is thus an identity map.  Thus, the
  only way an object $\bigl((f'\colon \alpha \to \gamma'),
  (1_{G\gamma})\bigr)$ with $f'$ being $G$-equivalent to $f$ can
  connect via a zig-zag to an object $\bigl((h\colon \alpha \to
  \epsilon), (G\epsilon \to G\gamma)\bigr)$ with $h$ not
  $G$-equivalent to $f$ is by way of a map $\tau\colon \epsilon \to
  \gamma'$ from an object $\bigl((h\colon \alpha \to \epsilon),
  (G\epsilon \to G\gamma)\bigr)$ with $\degree(G\epsilon) >
  \degree(G\gamma)$, which (by the induction hypothesis) implies that
  $h\colon \alpha \to \epsilon$ is controlled.  In this case, the
  composition $\alpha \xrightarrow{h} \epsilon \xrightarrow{\tau}
  \gamma'$ equals $f'\colon \alpha \to \gamma'$, and so $f\colon
  \alpha \to \gamma$ is controlled.  This completes the induction.
\end{proof}

\begin{proof}[Proof of \propref{prop:MatchIso}]
  \propref{prop:DetectIso} implies that it is sufficient to show that
  for every object $W$ of $\cat M$ the matching map $(G^{*}\diag
  X)_{\alpha} \to \match^{\cat C}_{\alpha}(G^{*}\diag X)$ induces an
  isomorphism of the sets of maps
  \begin{equation}
    \label{eq:MatchIso}
    \xymatrix{
      {\cat M\bigl(W, (G^{*}\diag X)_{\alpha}\bigr)} \ar[r]^-{\iso}
      & {\cat M\bigl(W, \match^{\cat C}_{\alpha}(G^{*}\diag X)\bigr)
         \Period}
    }
  \end{equation}
  Let $W$ be an object of $\cat M$ and let $h\colon W \to \match^{\cat
    C}_{\alpha}(G^{*}\diag X)$ be a map.  If $\alpha \to \gamma$ is an
  object of $\matchcat{\cat C}{\alpha}$ that is in the $G$-kernel at
  $\alpha$, then $(G^{*}\diag X)_{(\alpha \to \gamma)} = (G^{*}\diag
  X)_{\gamma} = (G^{*}\diag X)_{\alpha}$, and so the projection of $h$
  onto $(G^{*}\diag X)_{(\alpha \to \gamma)}$ defines a map
  $\hhat\colon W \to (G^{*}\diag X)_{\alpha}$.  \lemref{lem:IdDiag}
  implies that the map $\hhat$ is independent of the choice of object
  of the $G$-kernel at $\alpha$.

  The composition
  \begin{displaymath}
    \xymatrix{
      {W} \ar[r]^-{\hhat}
      & {(G^{*}\diag X)_{\alpha}} \ar[r]
      & {\match^{\cat C}_{\alpha}(G^{*}\diag X)}
    }
  \end{displaymath}
  has the same projection onto $(G^{*}\diag X)_{(\alpha\to\gamma)}$ as
  the map $h\colon W \to \match^{\cat C}_{\alpha}(G^{*}\diag X)$;
  since every object of $\matchcat{\cat C}{\alpha}$ is controlled (see
  \propref{prop:AllControlled}), these two maps agree on every
  projection of $\match^{\cat C}_{\alpha}(G^{*}\diag X)$ (see
  \lemref{lem:controlled}), and so they are equal; thus, the map
  \eqref{eq:MatchIso} is a surjection. Since the composition of the
  matching map with the projection $\match^{\cat
    C}_{\alpha}(G^{*}\diag X) \to (G^{*}\diag X)_{(\alpha \to
    \gamma)}$ is $\diag X \circ G$ applied to $\alpha\to\gamma$, which
  is the identity map, $\hhat$ is the only possible lift to
  $(G^{*}\diag X)_{\alpha}$ of $h$, and so the map \eqref{eq:MatchIso}
  is also an injection, and so it is an isomorphism.
\end{proof}

\subsection{Proof of \thmref{thm:RtGdNec}}
\label{sec:RtGdNec}

We will begin by constructing the $\cat D$-diagram $\diag X$ whose
existence is asserted in \thmref{thm:RtGdNec}.  The construction is by
induction on the filtrations $\F^{n}\cat D$ of $\cat D$ (see
\defref{def:filtration}), and it will follow immediately that
$\diag X$ is a fibrant $\cat D$-diagram (see \propref{prop:FibD}).
\propref{prop:ProdInts} will then describe the diagram $\diag X$ in
more detail.

\propref{prop:MatchProd} describes the matching object
$\match^{\cat C}_{\alpha}(G^{*}\diag X)$ of the induced
$\cat C$-diagram $G^{*}\diag X$ at an object $\alpha$ of $\cat C$, and
then \propref{prop:NotFib} shows that the matching map
$(G^{*}\diag X)_{\alpha} \to \match^{\cat C}_{\alpha}(G^{*}\diag X)$
is not a fibration, which implies that $G^{*}\diag X$ is not fibrant.
This plan is illustrated in the following diagram:
\begin{equation}
  \label{eq:RtGdNecFlow}
  \vcenter{
    \xymatrix@C=1.5em{
   &  \text{\thmref{thm:RtGdNec}} \\
   \text{\propref{prop:FibD}}\ar@{=>}[ur]
      & \text{\propref{prop:NotFib}}\ar@{=>}[u] \\
       \text{\propref{prop:ProdInts}} \ar@{=>}[r]\ar@{=>}[ur]
      & \text{\propref{prop:MatchProd}}\ar@{=>}[u]
        }
  }
\end{equation}

Our $\cat D$-diagram $\diag X$ will be a diagram in the standard model
category of topological spaces.  Throughout its construction, the
reader should keep the square diagram from \exref{ex:NotGood} in mind.
In that example, the diagram $\diag X$ that we construct here is the
functor that sends each object in that square to the unit interval $I$
with all the maps going to the identity map, and
$G\colon \cat C \to \cat D$ is the inclusion of the diagram obtained
by removing the degree zero object $\beta$ from the square.

To construct the $\cat D$-diagram $\diag X$ we set the object
$\diag X_{\beta}$ (for a particular object $\beta$ of $\cat D$) equal
to the unit interval $I$ (see the construction below), and then
\propref{prop:ProdInts} shows that for every object $\gamma$ of
$\cat D$ the space $\diag X_{\gamma}$ is a product of copies of $I$.
We remark that there is nothing essential about the choice of the
space $I$; it could be replaced by any space $Y$ that is path
connected and has more then one point (see the proof of
\propref{prop:NotFib}).
 
We will define the diagram $\diag X$ inductively over the filtrations
$\F^{n}\cat D$ of $\cat D$ (see \defref{def:filtration} and
\propref{prop:ConstructFilt}).  To start this inductive construction,
since $G\colon \cat C \to \cat D$ is not a fibering Reedy functor,
there are objects $\alpha\in \Ob(\cat C)$ and $\beta\in \Ob(\cat D)$
and a map $\sigma\colon G\alpha \to \beta$ in $\inv{\cat D}$ such that
the nerve of the category of inverse $\cat C$-factorizations of
$(\alpha, \sigma)$ (see \defref{def:CFactors}) is nonempty and not
connected.  Let $n_{\beta}$ be the degree of $\beta$. We have two
cases:
\begin{itemize}
\item If $n_{\beta} = 0$, we begin by letting $\diag X\colon
  \F^{0}\cat D \to \Top$ take $\beta$ to the unit interval $I$ and all
  other objects of $\F^{0}\cat D$ to $*$ (the one-point space).

\item If $n_{\beta} > 0$, we begin by letting $\diag X\colon
  \F^{(n_{\beta})-1}\cat D \to \Top$ be the constant functor at $*$
  (the one-point space).  Then, to extend $\diag X$ from
  $\F^{(n_{\beta})-1}\diag D$ to $\F^{n_{\beta}}\diag D$, we let
  $\diag X_{\beta} = I$, the unit interval.  We factor
  $\latch_{\beta}\diag X \to \match_{\beta}\diag X$ as
  \begin{displaymath}
    \latch_{\beta}\diag X \longrightarrow I \longrightarrow
    \match_{\beta}\diag X
  \end{displaymath}
  where the first map is the constant map at $0 \in I$ and the second
  map is the unique map $I \to *$ (since $\diag X_{\gamma} = *$ is the
  terminal object of $\Top$ for all objects $\gamma$ of degree less
  than $n_{\beta}$, that matching object is $*$).  If $\gamma$ is any
  other object of $\cat D$ of degree $n_{\beta}$, we let $\diag
  X_{\gamma} = \match_{\gamma}\diag X$ and let $\latch_{\gamma}\diag X
  \to \diag X_{\gamma} \to \match_{\gamma}\diag X$ be the natural map
  followed by the identity map.
\end{itemize}
We now define $\diag X\colon F^{n}\cat D \to \Top$ for $n > n_{\beta}$
inductively on $n$ by letting $\diag X_{\gamma} = \match_{\gamma}\diag
X$ for every object $\gamma$ of degree $n$ and letting the
factorization $\latch_{\gamma}\diag X \to \diag X_{\gamma} \to
\match_{\gamma}\diag X$ be the natural map followed by the identity
map.

\begin{prop}
  \label{prop:FibD}
  The $\cat D$-diagram of topological spaces $\diag X$ is fibrant.
\end{prop}

\begin{proof}
  The matching map at the object $\beta$ of $\cat D$ is the map $I \to
  *$, which is a fibration, and the matching map at every other object
  of $\cat D$ is an identity map, which is also a fibration.
\end{proof}

We now give a more detailed description of the diagram $\diag X$.
\begin{prop}
  \label{prop:ProdInts}
  \leavevmode
  \begin{enumerate}
  \item For every object $\gamma$ in $\cat D$ the space $\diag
    X_{\gamma}$ is homeomorphic to a product of unit intervals, one
    for each map $\gamma \to \beta$ in $\inv{\cat D}$ (and so, for
    objects $\gamma$ for which there are no maps $\gamma \to \beta$ in
    $\inv{\cat D}$, the space $\diag X_{\gamma}$ is the empty product,
    and is thus equal to the terminal object, the one-point space
    $*$).
  \item Under the isomorphisms of part~1, if $\tau\colon \gamma \to
    \delta$ is a map in $\inv{\cat D}$, then the projection of $\diag
    X_{\tau}\colon \diag X_{\gamma} \to \diag X_{\delta}$ onto the
    factor $I$ of $\diag X_{\delta}$ indexed by a map $\mu\colon
    \delta \to \beta$ in $\inv{\cat D}$ is the projection of $\diag
    X_{\gamma}$ onto the factor $I$ of $\diag X_{\gamma}$ indexed by
    $\mu\tau\colon \gamma \to \beta$.
  \end{enumerate}
\end{prop}

\begin{proof}
  We will use an induction on $n$ to prove both parts of the
  proposition simultaneously for the restriction of $\diag X$ to each
  filtration $\F^{n}\cat D$ of $\cat D$.  The induction is begun at $n
  = n_{\beta}$ because the only map in $\F^{n_{\beta}}\inv{\cat D}$ to
  $\beta$ is the identity map of $\beta$, the only object of
  $\F^{n_{\beta}}\inv{\cat D}$ at which $\diag X$ is not a single
  point is $\beta$, and $\diag X_{\beta} = I$.

  Suppose now that $n > n_{\beta}$, the statement is true for the
  restriction of $\diag X$ to $\F^{n-1}\cat D$, and that $\gamma$ is
  an object of degree $n$.  The space $\diag X_{\gamma}$ is defined to
  be the matching object $\match_{\gamma}\diag X =
  \lim_{\matchcat{\cat D}{\gamma}} \diag X$.  There is a discrete
  subcategory ${\cat E}_{\gamma}$ of the matching category
  $\matchcat{\cat D}{\gamma}$ consisting of the maps $\gamma \to
  \beta$ in $\inv{\cat D}$, and so there is a projection map
  \begin{displaymath}
    \match_{\gamma}\diag X =
    \lim_{\matchcat{\cat D}{\gamma}} \hspace{-0.5em}\diag X
    \longrightarrow \lim_{\cat E_{\gamma}} \diag X =
    \prod_{(\gamma \to \beta)\in\inv{\cat D}}
    \hspace{-1em}\diag X_{\beta} =
    \prod_{(\gamma \to \beta)\in\inv{\cat D}} \hspace{-1em}I \Period
  \end{displaymath}
  We will show that that projection map $p\colon \lim_{\matchcat{\cat
      D}{\gamma}} \diag X \to \prod_{(\gamma\to \beta)\in\inv{\cat D}}
  I$ is a homeomorphism by defining an inverse homeomorphism
  \begin{displaymath}
    q\colon \prod_{(\gamma\to \beta)\in\inv{\cat D}} \hspace{-1em}I
    \longrightarrow \lim_{\matchcat{\cat D}{\gamma}}
    \hspace{-0.5em}\diag X  \Period
  \end{displaymath}

  We define the map $q$ by defining its projection onto $\diag
  X_{(\tau\colon\gamma\to\delta)} = \diag X_{\delta}$ for each object
  $(\tau\colon\gamma \to \delta)$ of $\matchcat{\cat D}{\gamma}$.  The
  induction hypothesis implies that $\diag X_{\tau} = \diag
  X_{\delta}$ is isomorphic to $\prod_{(\delta \to \beta)\in\inv{\cat
      D}} I$, and we let the projection onto the factor indexed by
  $\mu\colon \delta \to \beta$ be the projection of $\prod_{(\gamma\to
    \beta)\in\inv{\cat D}} I$ onto the factor indexed by
  $\mu\tau\colon \gamma \to \beta$.  To see that this defines a map to
  $\lim_{\matchcat{\cat D}{\gamma}} \diag X$, let $\nu\colon \delta
  \to \epsilon$ be a map from $\tau\colon \gamma \to \delta$ to
  $\nu\tau\colon \gamma \to \epsilon$ in $\matchcat{\cat D}{\gamma}$
  (see \diagref{diag:ProdInts}).
  The induction hypothesis implies that the projection of the map
  $\diag X_{\nu}\colon \diag X_{\tau} = \diag X_{\delta} \to \diag
  X_{\nu\tau} = \diag X_{\epsilon}$ onto the factor of $\diag
  X_{\epsilon}$ indexed by $\xi\colon \epsilon \to \beta$ in
  $\inv{\cat D}$ is the projection of $\diag X_{\tau} = \diag
  X_{\delta}$ onto the factor indexed by $\xi\nu\colon \delta \to
  \beta$.
  \begin{equation}
    \label{diag:ProdInts}
    \vcenter{
      \xymatrix@C=1em{
        &{\gamma} \ar[dl]_{\tau} \ar[dr]^{\nu\tau}\\
        {\delta} \ar[rr]_{\nu}
        && {\epsilon} \ar[rr]_{\xi}
        && {\beta}
      }
    }
  \end{equation}
  Thus, the projection of the composition
  $\prod_{(\gamma\to\beta)\in\inv{\cat D}} I \to \diag X_{\tau} =
  \diag X_{\delta} \xrightarrow{\diag X_{\nu}} \diag X_{\nu\tau} =
  \diag X_{\epsilon}$ onto the factor indexed by $\xi\colon \epsilon
  \to \beta$ equals the projection of
  $\prod_{(\gamma\to\beta)\in\inv{\cat D}} I$ onto the factor indexed
  by $\xi\nu\tau\colon \gamma \to \beta$, which equals that same
  projection of the map $\prod_{(\gamma\to\beta)\in\inv{\cat D}} I \to
  \diag X_{\nu\tau: \gamma \to \epsilon} = \diag X_{\epsilon}$.  Thus,
  we have defined the map $q$.

  It is immediate from the definitions that $pq$ is the identity map
  of $\prod_{(\gamma\to\beta)\in\inv{\cat D}} I$.  To see that $qp$ is
  the identity map of $\lim_{\matchcat{\cat D}{\gamma}} \diag X$, we
  first note that the definitions immediately imply that the
  projection of $qp$ onto each $\diag X_{(\gamma \to \beta)} = \diag
  X_{\beta}$ equals the corresponding projection of the identity map
  of $\lim_{\matchcat{\cat D}{\gamma}} \diag X$.  If $\tau\colon
  \gamma \to \delta$ is any other object of $\matchcat{\cat
    D}{\gamma}$, then the induction hypothesis implies that $\diag
  X_{\tau} = \diag X_{\delta}$ is homeomorphic to the product
  $\prod_{(\delta\to\beta)\in\inv{\cat D}}I$.  Every $\mu\colon \delta
  \to \beta$ in $\inv{\cat D}$ defines a map $\mu_{*}\colon
  (\tau\colon \gamma\to\delta) \to (\mu\tau\colon \gamma \to \beta)$
  in $\matchcat{\cat D}{\gamma}$, and the induction hypothesis implies
  that the map $\diag X_{\mu}\colon \diag X_{\tau} = \diag X_{\delta}
  \to \diag X_{\mu\tau} = \diag X_{\beta} = I$ is projection onto the
  factor indexed by $\mu$.  Thus, for any map to $\lim_{\matchcat{\cat
      D}{\gamma}} \diag X$, its projection onto $\diag X_{\tau} =
  \diag X_{\delta}$ is determined by its projections onto the $\diag
  X_{(\gamma\to\beta)\in\inv{\cat D}}$; since $qp$ and the identity
  map agree on those projections, $qp$ must equal the identity map.
  This completes the induction for part~1.

  For part~2, for every map $\tau\colon \gamma \to \delta$ in
  $\inv{\cat D}$ the map $\diag X_{\tau}\colon \diag X_{\gamma} \to
  \diag X_{\delta}$ equals the composition
  \begin{displaymath}
    \diag X_{\gamma} \longrightarrow
    \lim_{\matchcat{\cat D}{\gamma}} \hspace{-0.5em}\diag X
    \longrightarrow \diag X_{\delta}
  \end{displaymath}
  where the first map is the matching map of $\diag X$ at $\gamma$ and
  the second is the projection from the limit $\lim_{\matchcat{\cat
      D}{\gamma}} \diag X \to \diag X_{(\tau\colon \gamma \to \delta)}
  = \diag X_{\delta}$ (this is the case for \emph{every} $\cat
  D$-diagram in $\cat M$, not just for $\diag X$).  Since the matching
  map at every object other than $\beta$ is the identity map, the map
  $\diag X_{\tau}\colon \diag X_{\gamma} \to \diag X_{\delta}$ is the
  projection $\lim_{\matchcat{\cat D}{\gamma}}\diag X \to \diag
  X_{(\tau\colon \gamma \to \delta)} = \diag X_{\delta}$.  The
  discussion in the previous paragraph shows that the projection of
  $\diag X_{\tau}\colon \diag X_{\gamma} \to \diag X_{\delta}$ onto
  the factor of $\diag X_{\delta}$ indexed by $\mu\colon \delta \to
  \beta$ is the projection of $\diag X_{\gamma}$ onto the factor
  indexed by $\mu\tau\colon \gamma \to \beta$.  This completes the
  induction for part~2.
\end{proof}

We now consider the diagram $G^{*}\diag X$ that $G\colon \cat C \to
\cat D$ induces on $\cat C$ from $\diag X$.
\begin{prop}
  \label{prop:MatchProd}
  The matching object $\match^{\cat C}_{\alpha} G^{*}\diag X =
  \lim_{\matchcat{\cat C}{\alpha}} G^{*}\diag X$ of the induced
  diagram on $\cat C$ at $\alpha$ is homeomorphic to a product of unit
  intervals indexed by the union over the maps $\tau\colon G\alpha \to
  \beta$ in $\inv{\cat D}$ of the sets of path components of the nerve
  of the category of inverse $\cat C$-factorizations of $(\alpha,
  \tau)$.  That is,
  \begin{displaymath}
    \match^{\cat C}_{\alpha} G^{*}\diag X \iso
    \prod_{(\tau\colon G\alpha\to\beta)\in\inv{\cat D}}
    \Biggl(\prod_{\pi_{0}\N(\invfact{\cat C}{\alpha}{\tau})}
      I \Biggr)
    \Period
  \end{displaymath}
\end{prop}

\begin{proof}
  Let $S = \coprod_{(\alpha \to \gamma)\in\Ob(\matchcat{\cat
      C}{\alpha})} \inv{\cat D}(G\gamma,\beta)$, the disjoint union
  over all objects $\alpha\to\gamma$ of $\matchcat{\cat C}{\alpha}$ of
  the set of maps $\inv{\cat D}(G\gamma,\beta)$.  An element of $S$ is
  then an ordered pair $\bigl((\nu\colon \alpha \to \gamma),
  (\mu\colon G\gamma \to \beta)\bigr)$ where $\nu\colon \alpha \to
  \gamma$ is an object of $\matchcat{\cat C}{\alpha}$ and $\mu\colon
  G\gamma \to \beta$ is a map in $\inv{\cat D}$, and is thus an object
  of the category of inverse $\cat C$-factorizations of the
  composition $(\alpha, G\alpha \xrightarrow{G\nu} G\gamma
  \xrightarrow{\mu} \beta)$, i.e., of $(\alpha, \mu\circ G\nu\colon
  G\alpha \to\beta)$.  Every object of the category of inverse $\cat
  C$-factorizations of every map $(\alpha, \tau\colon G\alpha \to
  \beta)$ in $\inv{\cat D}$ appears exactly once, and so the set $S$
  is the union over all maps $\tau\colon G\alpha \to \beta$ in
  $\inv{\cat D}$ of the set of objects of the category of inverse
  $\cat C$-factorizations of $(\alpha, \tau)$.

  \propref{prop:ProdInts} implies that for every object $\tau\colon
  \alpha \to \gamma$ in $\matchcat{\cat C}{\alpha}$ the space
  $(G^{*}\diag X)_{\tau} = (G^{*}\diag X)_{\gamma} = \diag
  X_{G\gamma}$ is a product of unit intervals, one for each map
  $G\gamma \to \beta$ in $\inv{\cat D}$, and so the product over all
  objects $\tau\colon \alpha\to\gamma$ of $\matchcat{\cat C}{\alpha}$
  of $(G^{*}\diag X)_{\tau} = (G^{*}\diag X)_{\gamma} = \diag
  X_{G\gamma}$ is homeomorphic to the product of unit intervals
  indexed by $S$, i.e.,
  \begin{displaymath}
    \prod_{(\alpha\to\gamma)\in\Ob(\matchcat{\cat C}{\alpha})}
    \hspace{-3em} (G^{*}\diag X)_{\gamma}
    \iso
    \prod_{S} I \Period
  \end{displaymath}
  The matching object $\match^{\cat C}_{\alpha} G^{*}\diag X$ is a
  subspace of that product.  More specifically, it is the subspace
  consisting of the points such that, for every map
  \begin{displaymath}
    \xymatrix@=.6em{
      & {\alpha} \ar[dl]_{\nu} \ar[dr]^{\nu'}\\
      {\gamma} \ar[rr]_{\tau}
      && {\gamma'}
    }
  \end{displaymath}
  in $\matchcat{\cat C}{\alpha}$ from $\nu\colon \alpha \to \gamma$ to
  $\nu'\colon \alpha \to \gamma'$ and every map $\mu'\colon G\gamma'
  \to \beta$ in $\inv{\cat D}$, the projection onto the factor indexed
  by $\bigl((\nu'\colon \alpha \to \gamma'), (\mu'\colon G\gamma' \to
  \beta)\bigr)$ equals the projection onto the factor indexed by
  $\bigl((\nu\colon \alpha \to \gamma), (\mu'\circ(G\tau)\colon
  G\gamma \to \beta)\bigr)$.

  Generate an equivalence relation on $S$ by letting $\bigl((\nu\colon
  \alpha \to \gamma), (\mu\colon G\gamma \to \beta)\bigr)$ be
  equivalent to $\bigl((\nu'\colon \alpha \to \gamma'), (\mu'\colon
  G\gamma' \to \beta)\bigr)$ if there is a map $\tau\colon \gamma \to
  \gamma'$ in $\inv{\cat C}$ such that $\tau\nu = \nu'$ and $\mu'\circ
  (G\tau) = \mu$, i.e., if there is a map in the category of inverse
  $\cat C$-factorizations of $(\alpha, \mu \circ (G\nu)\colon G\alpha
  \to \beta)$ from $\bigl((\nu\colon \alpha\to\gamma), (\mu\colon
  G\gamma \to \beta)\bigr)$ to $\bigl((\mu'\colon \alpha\to\gamma'),
  (\mu'\colon G\gamma' \to\beta)\bigr)$; let $T$ be the set of
  equivalence classes.  This makes two objects in the category of
  inverse $\cat C$-factorizations of a map equivalent if there is a
  zig-zag of maps in that category from one to the other, i.e., if
  those two objects are in the same component of the nerve, and so the
  set $T$ is the disjoint union over all maps $\tau\colon G\alpha \to
  \beta$ in $\inv{\cat D}$ of the set of components of the nerve of
  the category of inverse $\cat C$-factorizations of $(\alpha, \tau)$,
  i.e.,
  \begin{displaymath}
    T = \coprod_{(\tau\colon G\alpha\to\beta)\in\inv{\cat D}}
    \hspace{-1.5em}\pi_{0}\N\bigl(\invfact{\cat C}{\alpha}{\tau}\bigr)
    \Period
  \end{displaymath}

  Let $T'$ be a set of representatives of the equivalence classes $T$
  (i.e., let $T'$ consist of one element of $S$ from each equivalence
  class); we will show that the composition
  \begin{displaymath}
    \xymatrix{
      {\match^{\cat C}_{\alpha}G^{*}\diag X} \ar[r]^-{\subset}
      & {\prod_{S} I} \ar[r]^{p'}
      & {\prod_{T'} I}
    }
  \end{displaymath}
  (where $p'$ is the projection) is a homeomorphism.  We will do that
  by constructing an inverse $q\colon \prod_{T'} I \to \match^{\cat
    C}_{\alpha} G^{*}\diag X$ to the map $p\colon \match^{\cat
    C}_{\alpha} G^{*}\diag X \to \prod_{T'}I$ (where $p$ is the
  restriction of $p'$ to $\match^{\cat C}_{\alpha}G^{*}\diag X$).

  We first construct a map $q'\colon \prod_{T'} I \to \prod_{S} I$ by
  letting the projection of $q'$ onto the factor indexed by $s \in S$
  be the projection of $\prod_{T'}I$ onto the factor indexed by the
  unique $t \in T'$ that is equivalent to $s$.  The description above
  of the subspace $\match^{\cat C}_{\alpha}G^{*}\diag X$ of $\prod_{S}
  I$ makes it clear that $q'$ factors through $\match^{\cat
    C}_{\alpha}G^{*}\diag X$ and thus defines a map $q\colon
  \prod_{T'}I \to \match^{\cat C}_{\alpha} \diag X$.

  The composition $pq$ equals the identity of $\prod_{T'}I$ because
  the composition $p'q'$ equals the identity of $\prod_{T'}I$.  To see
  that the composition $qp$ equals the identity of $\match^{\cat
    C}_{\alpha}G^{*}\diag X$, it is sufficient to see that the
  projection of $qp$ onto the factor $I$ indexed by every element $s$
  of $S$ agrees with that of the identity map of $\match^{\cat
    C}_{\alpha}G^{*}\diag X$.  Since the projections of points in
  $\match^{\cat C}_{\alpha}G^{*}\diag X$ onto factors indexed by
  equivalent elements of $S$ are equal, and it is immediate that the
  projection of $\match^{\cat C}_{\alpha}G^{*}\diag X$ onto a factor
  indexed by an element of the set of representatives $T'$ agrees with
  the corresponding projection of $qp$, the projections for every
  element of $S$ must agree, and so $qp$ equals the identity of
  $\prod_{T'}I$.
\end{proof}

\begin{prop}
  \label{prop:NotFib}
  The diagram $G^{*}\diag X$ induced on $\cat C$ is not a fibrant
  $\cat C$-diagram.
\end{prop}

\begin{proof}
  We will show that the matching map $(G^{*}\diag X)_{\alpha} \to
  \match^{\cat C}_{\alpha} G^{*}\diag X$ of the induced $\cat
  C$-diagram at $\alpha$ is not a fibration.  Since the matching
  object $\match^{\cat C}_{\alpha} G^{*}\diag X$ is a product of unit
  intervals (see \propref{prop:MatchProd}), it is path connected, and
  so if the matching map were a fibration, it would be surjective.  We
  will show that the matching map is not surjective.

  Since $\sigma\colon G\alpha \to \beta$ is a map in $\inv{\cat D}$
  such that the nerve of the category of inverse $\cat
  C$-factorizations of $(\alpha, \sigma)$ is not connected, we can
  choose objects $(\nu\colon \alpha \to \gamma, \mu\colon G\gamma \to
  \beta)$ and $(\nu'\colon \alpha \to \gamma', \mu'\colon G\gamma' \to
  \beta)$ of that category that represent different path components of
  that nerve.  Since $\mu\circ (G\nu) = \mu'\circ (G\nu')$,
  \propref{prop:ProdInts} implies that the projection of the matching
  map onto the copies of $I$ indexed by those objects are equal, and
  so the projection onto the $I\times I$ indexed by that pair of
  components factors as the composition $\diag X_{\alpha} \to I \to
  I\times I$, where that second map is the diagonal map and is thus
  not surjective.
\end{proof}

\begin{proof}[Proof of \thmref{thm:RtGdNec}]
  This follows from \propref{prop:FibD} and \propref{prop:NotFib}.
\end{proof}

\subsection{Proof of \thmref{thm:MainCofibering}}
\label{sec:PrfCofibering}

Since $\cat M$ is complete, the right adjoint of $G^{*}$ exists and
can be constructed pointwise (see \cite{borceux-I}*{Thm.~3.7.2} or
\cite{McL:categories}*{p.~235}), and \thmref{thm:GoodisGood} implies
that $(G\op)^{*}\colon (\cat M\op)^{\cat D\op} \to (\cat M\op)^{\cat
  C\op}$ is a right Quillen functor for every model category $\cat
M\op$ if and only if $G\op$ is fibering (because every model category
$\cat N$ is of the form $\cat M\op$ for $\cat M = \cat N\op$).

\propref{prop:OpFiberingCofibering} implies that the functor $G\colon
\cat C \to \cat D$ is cofibering if and only if its opposite
$G\op\colon \cat C\op \to \cat D\op$ is fibering, and
\thmref{thm:GoodisGood} implies that this is the case if and only if
$(G\op)^{*}\colon (\cat M\op)^{\cat D\op} \to (\cat M\op)^{\cat C\op}$
is a right Quillen functor for every model category $\cat M\op$, which
is the case if and only if $G^{*}\colon \cat M^{\cat D} \to \cat
M^{\cat C}$ is a left Quillen functor for every model category $\cat
M$ (see \propref{prop:OpQuillen} and \propref{prop:OppositeReedy}).
\qed



\end{document}